\documentclass[12pt,reqno]{amsart}
\usepackage{amssymb}
\usepackage{graphicx}
\usepackage{amsmath}
\usepackage{a4wide}
\usepackage{version}
\usepackage{mathrsfs}  
\usepackage{tikz}
\usepackage{mathtools}
\usetikzlibrary{arrows,decorations.markings, matrix}
\usepackage{hyperref}

\usepackage{float} 
\usepackage[all]{xy}

\numberwithin{equation}{section}

\newtheorem{thm}{Theorem}[subsection]
\newtheorem{lem}[thm]{Lemma}
\newtheorem{prop}[thm]{Proposition}
\newtheorem{cor}[thm]{Corollary}

\theoremstyle{definition}
\newtheorem{definition}[thm]{Definition}
\newtheorem{example}[thm]{Example}
\newtheorem{rmk}[thm]{Remark}

\newcommand{\Z}{\mathbb{Z}}

\newcommand{\K}{\mathbb{K}}

\newcommand{\id}{\operatorname{id}}

\newcommand{\ov}{\overline}

\newcommand{\MM}{{\mathcal M}}

\newcommand{\Hom}{\operatorname{Hom}}

\newcommand{\Ext}{\operatorname{Ext}}

\newcommand{\Aut}{\operatorname{Aut}}
\newcommand{\dgmod}{\operatorname{dgmod}}

\renewcommand{\a}{\alpha}
\renewcommand{\b}{\beta}

\newcommand{\La}{\Lambda}
\newcommand{\Ga}{\Gamma}

\newcommand{\wt}{\widetilde}
\newcommand{\ot}{\otimes}

\newcommand{\sub}{\subset}

\renewcommand{\mod}{\operatorname{mod}}

\newcommand{\RR}{{\mathcal R}}

\newcommand{\Si}{\Sigma}

\newcommand{\Quad}{\operatorname{Quad}}

\newcommand{\SL}{\operatorname{SL}}

\newcommand{\G}{{\mathbb G}}

\newcommand{\lan}{\langle}
\newcommand{\ran}{\rangle}
\newcommand{\Coh}{\operatorname{Coh}}

\renewcommand{\P}{{\mathbb P}}

\newcommand{\si}{\sigma}

\newcommand{\ga}{\gamma}

\newcommand{\im}{\operatorname{im}}

\newcommand{\A}{{\mathbb A}}

\title{Derived equivalences of gentle algebras via Fukaya categories}
\author{Yank\i\ Lekili} 
\author{Alexander Polishchuk} 
\address{King's College London}
\address{University of Oregon, National Research University, Higher School of Economics Moscow, Russia, Korea Institute for Advanced Study, Seoul, South Korea}
\begin{document}

\begin{abstract} Following the approach of Haiden-Katzarkov-Kontsevich \cite{HKK}, 
to any homologically smooth $\Z$-graded gentle algebra $A$ we associate a triple 
$(\Sigma_A, \Lambda_A; \eta_A)$, where $\Sigma_A$ is an oriented smooth surface with non-empty boundary, 
$\Lambda_A$ is a set of stops on $\partial \Sigma_A$ and $\eta_A$ is a line field on $\Sigma_A$, 
such that the derived category of perfect dg-modules of $A$ is equivalent to the 
partially wrapped Fukaya category of $(\Sigma_A, \Lambda_A ;\eta_A)$. 
Modifying arguments of Johnson and Kawazumi, we classify the orbit 
decomposition of the action of the (symplectic) mapping class group of $\Sigma_A$ on the 
homotopy classes of line fields. 
As a result we obtain a sufficient criterion for homologically smooth graded gentle algebras to be derived equivalent. 
Our criterion uses numerical invariants generalizing those given by Avella-Alaminos-Geiss in \cite{AAG},
as well as some other numerical invariants. As an application, we find many new cases when the AAG-invariants
determine the derived Morita class. As another application, we establish some derived equivalences between
the stacky nodal curves considered in \cite{LP-pwrap}.
\end{abstract}

\maketitle

\section*{Introduction}

Given a Liouville manifold $(M, \omega =d \lambda)$, a rigorous definition of
the compact Fukaya category, $\mathcal{F}(M)$,  appears in the monograph
\cite{seidel}. This is a triangulated $A_\infty$-category linear over some base
ring $\mathbb{K}$.  Roughly speaking, the objects of $\mathcal{F}(M)$ are
compact, exact, oriented Lagrangian submanifolds in $M$, equipped with spin
structures (if $\mathrm{char} \mathbb{K} \neq 2$). The orientations on each
Lagrangian determine a $\mathbb{Z}_2$-grading on $\mathcal{F}(M)$, and the spin
structures enter in orienting the moduli spaces of holomorphic polygons that
enter into the definition of structure constants of the $A_\infty$ operations.
It is often convenient to upgrade the $\mathbb{Z}_2$-grading on
$\mathcal{F}(M)$ to a $\mathbb{Z}$-grading, which can be done under the
additional assumption that $2c_1(M)=0$ (see \cite{kontsevich}, \cite{graded}).
Under this assumption, one defines a notion of a \emph{grading structure} on
$M$, and correspondingly considers only \emph{graded Lagrangians} as objects of
$\mathcal{F}(M)$, which now becomes a $\mathbb{Z}$-graded category. We refer to
\cite{graded} for these general notions. In this paper, we focus our attention
to the case where $M = \Sigma$ is a punctured (real) 2-dimensional surface,
equipped with an area form. A grading structure on $\Sigma$ can be concretely
described as a homotopy class of a section $\eta$ of the projectivized tangent bundle of
$\mathbb{P}(T\Sigma)$. Note that there is an effective $H^1(\Sigma)$'s worth of choices (see Sec. \ref{linefields}). 
A Lagrangian can be graded if the
winding number of $\eta$ along $L$ vanishes, and in such a situation a grading
is a choice of a homotopy from the tangent lift $L \to TL \subset T\Sigma$ to
$\eta_{|L}$ along $L$. These gradings extend in a straightforward
manner to the wrapped Fukaya category $\mathcal{W}(\Sigma)$ which contains
$\mathcal{F}(\Sigma)$ as a full subcategory, but also allows non-compact
Lagrangians in $\Sigma$ and more generally, partially wrapped category
$\mathcal{W}(\Sigma, \Lambda)$, as studied in \cite[Sec. 2.1]{HKK}, where
$\Sigma$ is a surface with boundary and $\Lambda$ is a collection of stops (i.e., marked points) on
$\partial \Sigma$.  

By a {\it graded surface} $(\Si,\La;\eta)$ we mean an oriented surface with boundary $\Si$, together with a set $\La$ of
marked points on the boundary and a line field $\eta$.
Given two graded surfaces with stops, $(\Sigma_i, \Lambda_i; \eta_i)$ for $i=1,2$, 
and a homeomorphism $\phi: \Sigma_1 \to \Sigma_2$, such that 
$\phi(\Lambda_1)=\Lambda_2$, and $\phi_*(\eta_1)$ is homotopic to $\eta_2$ (we refer to such homeomorphisms
as {\it graded}), 
one gets an equivalence between the partially wrapped Fukaya categories 
$\mathcal{W}(\Sigma_1, \Lambda_1;\eta_1)$ and $\mathcal{W}(\Sigma_2, \Lambda_2; \eta_2)$. 
Thus, it is important to have a set of explicit computable invariants of a line field $\eta$
on a surface with boundary that determine the orbit of $\eta$ under
the action of the mapping class group of $\Sigma$.
Our first result (see Theorem \ref{lineorbit}) gives such invariants in terms of winding numbers of $\eta$.
In the most interesting case when genus is $\ge 2$, 
the invariants consist of the winding numbers along all the boundary components,
plus two more invariants, each taking values $0$ and $1$. The first of them 
decides whether the line field $\eta$ is induced by a non-vanishing vector field, while the second is the Arf-invariant of a certain
quadratic form over $\Z_2$. The cases of genus $1$ and $0$ are special due to the special nature of the corresponding
mapping class groups. In the case of genus $1$ there is a certain $\Z$-valued invariant in addition to the winding numbers
along boundary components.
Note that from the numerical invariants of Theorem \ref{lineorbit} one can also recover the genus of
the surface and the numbers of stops on the boundary components, so if these invariants match then the corresponding partially wrapped Fukaya categories are equivalent.

Next, we use this result to construct derived equivalences between {\it gentle algebras}, introduced by Assem and Skowr\'onski in \cite{AS}. This is a remarkable class algebras with monomial quadratic relations of special kind with a well understood
structure of indecomposable modules. Furthermore, their derived categories of modules also enjoy many nice properties 
(see \cite{BD-gentle} and references therein). 
Avella-Alaminos and Geiss \cite{AAG} gave a combinatorial definition of derived invariants
of finite-dimensional gentle algebras, which form a collection of pairs of non-negative integers $(m,n)$ with multiplicities.
We refer to these as {\it AAG-invariants}.
It is known that these invariants do not completely determine the derived Morita class of a gentle algebra in general (for example, see \cite{amiot}).

We consider $\Z$-graded gentle algebras and their perfect derived categories (the classical case corresponds to algebras concentrated in degree $0$). 
For such an algebra $A$, 
we denote by $D(A)$ the perfect derived category of dg-modules over 
$A$ viewed as a dg-algebra with zero differential. The category $D(A)$ has a natural dg-enhancement which we take into
account when talking about equivalences involving $D(A)$.

The connection between graded gentle algebras and Fukaya categories was established 
by Haiden, Katzarkov and Kontsevich in \cite{HKK} (cf. \cite{bocklandt}): they constructed collections of formal generators in some partially
wrapped Fukaya categories whose endomorphism algebras are graded gentle algebras.
In Theorem \ref{gentle-Fuk-thm} we give an inverse construction\footnote{The existence of such construction 
is mentioned in \cite{HKK}}: 
starting from a homologically smooth graded gentle algebra $A$ we construct a graded surface with stops
$(\Sigma_A, \Lambda_A; \eta_A)$ together with a set formal generators whose endomorphism algebra is
isomorphic to $A$. This leads to an equivalence of the partially wrapped Fukaya category $\mathcal{W}(\Sigma, \Lambda)$ with the derived category $D(A)$.
In addition, we generalize the combinatorial definition of AAG-invariants to possibly infinite-dimensional
graded gentle algebras and show
that they can be recovered from the winding numbers of $\eta_A$ along all boundary components.

Now recalling our numerical invariants of graded surfaces with stops from Theorem \ref{lineorbit} we obtain a sufficient
criterion for derived equivalence between homologically smooth graded gentle algebras.
Namely, if we start with two such algebras $A$ and $A'$ and find that 
the corresponding invariants from Theorem \ref{lineorbit}, determined by winding numbers of $\eta_A$ and $\eta_{A'}$, coincide
then we get a derived equivalence between $A$ and $A'$. 
More precisely, the first step is to check that $A$ and $A'$ have the same AAG-invariants.
In the case of genus $0$, this suffices. For genus $1$, one has to compute a certain invariant with values in $\Z_{\ge 0}$,
while for genus $>1$ one has to check that two invariants
with values in $\{0,1\}$ match. Note that the genus can be computed from the AAG-invariants.

As an application, using the above approach we obtain a sufficient criterion for derived
equivalence of homologically smooth graded gentle algebras
given purely in terms of AAG-invariants (see Corollary \ref{AAG-cor}).
Using Koszul duality, we also get a sufficient criterion
for derived equivalence of finite-dimensional gentle algebras with grading in degree $0$ 
(see Corollary \ref{AAG-fd-cor}).

In a different direction, we construct derived equivalences between stacky nodal curves studied in \cite{LP-pwrap}. Namely, these are either chains or rings of weighted projective lines glued to form stacky nodes,
locally modelled by quotients $(xy=0)/(x,y)\sim (\zeta^k x,\zeta y)$, where $\zeta^r=1$ and $k\in (\Z/r)^*$. 
In \cite[Thm.\ B]{LP-pwrap} we constructed an equivalence of the derived category of coherent sheaves on
such a stacky curve with the partially wrapped Fukaya category of some graded surface with stops (this can be
viewed as an instance of homological mirror symmetry).
Thus, using Theorem \ref{lineorbit} we get many nontrivial derived equivalences between our stacky curves.
In the case of balanced nodes (those with $k=-1$) we recover the equivalences between tcnc curves from \cite{sibilla}.

{\it Acknowledgments}. Y.L. is partially supported by the Royal Society (URF) and
the NSF grant DMS-1509141, and would like to thank Martin Kalck for pointing out the reference \cite{AAG}. 
A.P. is supported in part by the NSF grant DMS-1700642 and by the Russian Academic Excellence Project `5-100'. 
While working on this project, A.P. was visiting King's College London, Institut des Hautes Etudes Scientifiques, and Korea Institute for Advanced Study. He would like to thank these institutions for hospitality and excellent working conditions.

\section{Line fields on surfaces}
\label{linefields} 

\subsection{Basics on line fields}
Let $\Sigma$ be an oriented smooth surface of genus $g(\Sigma)$ with non-empty boundary with connected components $\partial{\Sigma} = \bigsqcup_{i=1}^d \partial_i \Sigma$. The pure mapping class group of $\Sigma$ is 
\[ \mathcal{M}(\Sigma) = \pi_0(Homeo^+ (\Sigma, \partial{\Sigma})),\]
where $Homeo^+(\Sigma,\partial{\Sigma})$ is the space of orientation preserving homeomorphism of $\Sigma$ which are the identity pointwise on $\partial \Sigma$. 
\begin{definition}
An (unoriented) line field $\eta$ on $\Sigma$ is a section of the projectivized tangent bundle $\P(T\Sigma)$. We denote by 
\[ G(\Sigma) = \pi_0(\Gamma (\Sigma, \P(T\Sigma))) \]
the set of homotopy classes of unoriented line fields. 
\end{definition}

A non-vanishing vector field gives a section of the tangent circle bundle $\mathbb{S}\Sigma$. Such a section induces a line field via the bundle map 
$\mathbb{S}\Sigma \to \mathbb{P}(T\Sigma)$ which is a fibrewise double covering. 
However, not all line fields come from
non-vanishing vector fields: a section of $\mathbb{P}(T\Sigma)$ may not lift to a section of $\mathbb{S}\Sigma$
(in Lemma \ref{even-wind-numbers-lem} below we will get a criterion for this).

The trivial circle fibration 
\begin{align}\label{circle-fibration-eq}
S^1 \xrightarrow{\iota} \P(T\Sigma) \xrightarrow{p} \Sigma 
\end{align}
induces an exact sequence
\begin{align} \label{circle-fibration-ex-seq} 0 \to H^1(\Sigma) \xrightarrow{p^*} H^1(\P(T\Sigma))\xrightarrow{\iota^*} H^1(S^1)\to 0
\end{align}
(here and below, when the coefficient group is omitted it is assumed to be $\Z$).
Note that the orientation on $\Sigma$ induces orientations on the tangent circles, so that the inclusion $\iota$ used in the above sequence is canonical up to homotopy.

We can think of line fields as trivializations of the circle fibration \eqref{circle-fibration-eq}, in particular, the  
set $G(\Sigma)$ has a natural structure of a torsor over
the group of homotopy classes of maps $\Sigma\to S^1$, i.e., with $H^1(\Sigma)$.
We denote the corresponding
action of $c\in H^1(\Sigma)$ on $G(\Sigma)$ by $\eta\mapsto\eta+c$,  
 
Let us associate with a line field $\eta$ the class $[\eta]\in H^1(\P(T\Sigma))$, such that $\iota^*[\eta]([S^1])=1$, by taking the Poincar\'e-Lefschetz dual of the class of the image $[\eta(\Sigma)] \subset H_2(\P(T\Sigma), \partial \P(T\Sigma))$. 

\begin{lem} The map $\eta\mapsto [\eta]$
gives an identification
\[ G(\Sigma) = (\iota^*)^{-1}(\zeta) \subset H^1(\P(T\Sigma)),\]
where $\zeta \in H^1(S^1)$ is the generator which integrates to 1 along $S^1$. 
\end{lem}

\begin{proof} The exact sequence \eqref{circle-fibration-ex-seq} shows that
set $(\iota^*)^{-1}(\zeta)$ is a torsor over $H^1(\Sigma)$. It is easy to check that the map $\eta\mapsto [\eta]$
is compatible with the $H^1(\Sigma)$-actions, i.e.,
$$[\eta+c]=[\eta]+p^*c.$$
The assertion follows immediately from this.
\end{proof}

The mapping class group $\mathcal{M}(\Sigma)$ acts on $G(\Sigma)$ on the right. Our goal in this section is to understand the orbit decomposition of $G(\Sigma)$ with respect to this action. 

Given an immersed curve $\gamma : S^1 \to \Sigma$, one can consider its tangent lift $\tilde{\gamma}: S^1 \to \P(T\Sigma)$ given by $(\gamma, T\gamma)$, where $T\gamma$ is the tangent space to the curve $\gamma$. 

\begin{definition} Given a line field $\eta$ and an immersed curve $\gamma$, define the winding number of $\gamma$ with respect to $\eta$ to be \[ w_\eta(\gamma) := \langle [\eta], [\tilde{\gamma}] \rangle, \]
	where $\langle\ , \ \rangle :H^1(\P(T\Sigma)) \times H_1(\P(T\Sigma)) \to \Z$ is the natural pairing.
\end{definition}

The winding number $w_\eta(\gamma)$ with respect to $\eta$ only depends on the homotopy class of $\eta$ 
and the regular homotopy class of $\gamma$. From the definition we immediately get the following compatibility with
the action of $H^1(\Sigma)$:
$$w_{\eta+c}(\gamma)=w_{\eta}(\gamma)+\lan c,[\gamma]\ran.$$

Throughout, $\partial\Sigma$ is oriented with respect to the natural orientation as the boundary of $\Sigma$. In particular, $w_\eta(\partial \mathbb{D}^2) =2$ for the unique homotopy class of line fields on $\mathbb{D}^2$. For a boundary component $B \subset \partial\Sigma$ with the opposite orientation, we write $-B$. Then, we have $w_\eta(-B) = -w_\eta(B)$.

Every nonvanishing vector field $v$ on $\Sigma$ defines naturally a line field. In this way we get a map
$$V(\Sigma)\to G(\Sigma): v\mapsto \lan v\ran$$
from the set of homotopy classes of nonvanishing vector fields $V(\Sigma)$. We can think of nonvanishing vector fields
as trivializations of the tangent circle bundle, so $V(\Sigma)$ has a natural action of the group of homotopy classes of maps
$\Sigma\to S^1$, i.e., of $H^1(\Sigma)$. It is easy to check that the above map is compatible with the $H^1(\Sigma)$-actions
via the multiplication by $2$:
$$\lan v+c\ran=\lan v\ran + 2c$$
for $c\in H^1(\Sigma)$.
Also, for any nonvanishing vector field $v$, the winding number of the corresponding line field $\lan v\ran$ along an immersed curve $\ga$
is related to the winding number of $v$ itself by
$$w_{\lan v\ran}(\ga)=2 w_v(\ga).$$

\begin{lem}\label{even-wind-numbers-lem} 
A line field $\eta$ comes from a vector field if an only if all of its winding numbers are even.
\end{lem}

\begin{proof} The ``only if" part is clear. Now let $\eta$ be a line field with even winding numbers and let
$v$ be some nonvanishing vector field (it exists since $\Sigma$ is noncompact). Then $\eta=\lan v\ran+c$
for some $c\in H^1(\Sigma)$ such that $\lan c,[\ga]\ran$ is even for every homology class $[\ga]$. But this
implies that $c=2c'$, so $\eta=\lan v+c'\ran$.
\end{proof}

\subsection{Invariants under the action of the mapping class group}
Recall that $\partial_i \Sigma$, $i=1,\ldots,d$ are the components of the boundary of $\Sigma$. 
Given a line field $\eta$, the winding numbers  
$$w_\eta(\partial_i \Sigma)\ \text{ for } i=1,\ldots d,$$
depend only on the homotopy class of $\eta$ and are 
invariant under the action of the mapping class group $\mathcal{M}(\Sigma)$.
This gives us the first set of invariants of elements of $G(\Sigma)$. 

To go further, we need to study the winding numbers along non-separating curves on $\Sigma$. 
As is well-known, the winding number invariants do not descend to a map from $H_1(\Sigma)$. Indeed, if $S \subset \Sigma$ is a compact subsurface with boundary $\partial S = \bigsqcup_{i=1}^d \partial_i S$, by Poincar\'e-Hopf index theorem (see \cite[Ch. 3]{hopf}), we have:
\begin{equation}\label{chi-winding-eq} 
\sum_{i=1}^d w_\eta ( \partial_i S) = 2\chi(S) 
\end{equation}
However, considering the reduction modulo $2$ we still get a well-defined homomorphism (see \cite{johnson}):
\[ [w_\eta]^{(2)}: H_1(\Sigma; \Z_2) \to \Z_2 \]
i.e an element $H^1(\Sigma; \Z_2)$.

%

\begin{definition} 
We define the $\Z_2$-valued invariant 
\begin{align*} \sigma: \G(\Sigma) &\to \Z_2 \\
	\eta &\mapsto \begin{cases} 0 \text{ if } [w_\eta]^{(2)}=0 \\ 
	1 \text{ otherwise } \end{cases} 
\end{align*}
\end{definition}

We have a natural map induced by the inclusion $\partial \Sigma\to \Sigma$, 
\[ i: H_1(\partial \Sigma;\Z_2) \cong \Z_2^d \to H_1(\Sigma;\Z_2) \cong \Z_2^{2g+d-1}. \]
Note that the image of $i$ is precisely the kernel of the intersection pairing on $H_1(\Sigma,\Z_2)$, and
the induced pairing on the cokernel of $i$ is non-degenerate. In fact, this cokernel is naturally isomorphic to
$H_1(\ov{\Sigma};\Z_2)\simeq \Z_2^{2g}$, where $\ov{\Sigma}$ is the surface without boundary obtained from $\Sigma$ by
capping off all the boundary components.

Note that the values of $[w_\eta]^{(2)}$ on the boundary cycles are given by $w_\eta(\partial_i \Sigma)$ modulo $2$.
Thus, if at least one of these numbers is odd then $\sigma(\eta)=1$.
If all the boundary winding numbers are even then we can check whether $\sigma(\eta)=0$ by looking at 
the winding numbers of a collection of cycles that form a basis in the homology of $\ov{\Sigma}$.


%

\begin{prop}\label{quad-from-line-prop} 
Suppose $\eta$ is a line field on $\Sigma$ defined by the class $[\eta] \in H^1(\P(T\Sigma))$. 
There is a well defined map	
\[ q_\eta: H_1(\Sigma;\Z_4) \to \Z_4 \]
given by
	\[ q_\eta ( \sum_{i=1}^m \alpha_i) = \sum_{i=1}^m w_\eta(\alpha_i) + 2m  \in \Z_4, \]
where $\alpha_i$ are simple closed curves.
It satisfies
	\[ q_\eta(a+b) = q_\eta(a) + q_\eta(b) + 2 (a \cdot b) \in \Z_4 \]
where $a,b \in H_1(\Sigma;\Z_4)$,
and $a \cdot b$ denotes the intersection pairing on $H_1(\Sigma;\Z_4)$.
\end{prop}

\begin{proof} In the case when $\eta$ comes from a non-vanishing vector field $v$, we have
$w_\eta(a)=2w_v(a)$, where $w_v(\cdot)$ is the winding number of the vector field.
Hence, the assertion in this case follows from \cite[Thm 1A, Thm 1B]{johnson}.
In general, we have $\eta=\eta_0+c$, where $\eta_0$ comes from a non-vanishing vector field (which exists because $\Sigma$ is non-compact) and $c$ is a class in $H^1(\Sigma)$.
Thus, the function $q_\eta(a):=q_{\eta_0}(a)+\lan c,a\ran$ has the claimed properties.
\end{proof}

\begin{lem}\label{red-mod4-lem} Suppose that $g(\Sigma) \geq 2$. Assume that line fields $\eta$ and $\theta$ have
$w_\eta(\partial_i \Sigma)=w_\theta(\partial_i \Sigma)$ for $i=1,\ldots,d$, and $q_\eta=q_\theta$. Then their homotopy classes 
lie in the same $\MM(\Sigma)$-orbit.
\end{lem}

\begin{proof} The assumption $q_\eta=q_\theta$ implies that $w_\eta(a)\equiv w_\theta(a)\mod 4$ for any $a\in H_1(\Sigma)$.
Thus, we have $\theta=\eta+4c$ for some $c\in H^1(\Sigma)$. Furthermore, the condition 
$w_\eta(\partial_i \Sigma)=w_\theta(\partial_i \Sigma)$
implies that $c$ has zero restriction to $H_1(\partial \Sigma)$. Hence, there exists $\alpha\in H_1(\Sigma)$,
such that $\lan c,\gamma\ran=(\alpha\cdot\gamma)$ for any $\gamma\in H_1(\Sigma)$.
Now the fact that $\eta$ and $\theta$ lie in the same $\MM(\Sigma)$-orbit is proved in exactly the same way
as in the proof of \cite[Thm.\ 2.5]{kawazumi}. Namely, for each standard generator of the homology,
$\a$, one can construct an explicit element in the mapping class $f_\a$ (expressed in terms of Dehn twists along certain curves related to $\alpha$) such that
the action of $f_\a$ on a line field has the same effect as adding the class dual to $4\a$.
\end{proof}

Thus, for $g(\Sigma)\ge 2$, the study of the $\MM(\Sigma)$-orbits on $G(\Sigma)$ reduces to the study of $\MM(\Sigma)$-orbits on 
the set of functions 
$q:H_1(\Sigma,\Z_4)\to \Z_4$ satisfying
\begin{equation}\label{quadratic-main-eq}
q(a+b)=q(a)+q(b)+ 2 (a \cdot b).
\end{equation}
Let us denote by $\Quad_4=\Quad_4(\Sigma)$ the set of all such functions (it is an $H^1(\Sigma,\Z_4)$-torsor).

Recall that given a symplectic vector space $V, (- \cdot -)$ over $\Z_2$, one can consider the set $\Quad(V)$ of quadratic forms
$\ov{q}:V\to \Z_2$ satisfying 
\begin{equation}\label{quadratic-mod2-eq}
\ov{q}(x+y)=\ov{q}(x)+\ov{q}(y)+(x\cdot y).
\end{equation}
For every $\ov{q}\in\Quad(V)$, the Arf-invariant (\cite{arf},\cite{dickson}) is the element of $\Z_2$ given by
$$A(\ov{q})=\sum_{i=1}^n \ov{q}(a_i)\ov{q}(b_i),$$
where $(a_i,b_i)$ is a symplectic basis of $V$. The Arf invariant is the value that
$\ov{q}$ attains on the majority of vectors in $V$.

In the case when $w_\eta(\partial_i \Sigma)\equiv 2 \mod 4$ for every $i=1,\ldots,d$, and 
the quadratic function $q=q_\eta$ takes values in $2\Z_4$, we can associate to $q$ an element
in $\Quad(H_1(\ov{\Si},\Z_2))$ whose Arf-invariant will give us an additional
invariant of $\eta$ modulo the mapping class group action.

Namely, it is easy to see that if $q\in\Quad_4$ takes values in $2\Z_4$ then 
we have a well defined function $q/2:H_1(\Sigma,\Z_2)\to \Z_2$ satisfying \eqref{quadratic-mod2-eq}
such that $q=2\cdot q/2$. Now the condition $w_\eta(\partial_i\Sigma)\equiv 2\mod 4$ is equivalent to
$q/2(\partial_i\Sigma)=0$, so this is precisely the condition for the quadratic function $q/2$ to descend
to a form $\ov{q}$ in $\Quad(H_1(\ov{\Si},\Z_2))$ (recall that $H_1(\ov{\Si},\Z_2)$ is the quotient of
$H_1(\Si,\Z_2)$ by the boundary classes). 

Thus, in the case when $\si(\eta)=0$ and $w_\eta(\partial_i \Sigma)\equiv 2 \mod 4$ for every $i=1,\ldots,d$,
we can apply the above construction to $q_\eta$ and define the quadratic form $\ov{q}_\eta$ in
$\Quad(H_1(\ov{\Si},\Z_2))$. In this case we set
$$A(\eta):=A(\ov{q}_\eta).$$

In the case $g(\Sigma)=1$ we will use a different invariant of a line field, $\wt{A}(\eta)$, defined by
\begin{equation}
\label{genus-1-line-field-inv}
\wt{A}(\eta):= \gcd( \{ w_\eta(\alpha),w_\eta(\beta),w_\eta(\partial_1\Sigma)+2,\ldots,w_\eta(\partial_d \Sigma)+2\}),
\end{equation}
where $\alpha,\beta$ are non-separating curves in $\Sigma$
such that $[\alpha]$ and $[\beta]$ project to a basis of $H_1(\Sigma)/\im(i_*)$.
It can be shown as in \cite[Lemma 2.6]{kawazumi} that 
\[\wt{A}(\eta) = \gcd(\{ w_\eta(\gamma) : \gamma \text{\ non-separating\ }\}) \]
which implies that $\wt{A}(\cdot)$ is indeed invariant under the mapping class group.
We also note that in the case $d=1$, $w_\eta(\partial\Sigma)=-2$, hence this invariant reduces to $\gcd(w_\eta(\alpha), w_\eta(\beta))$ considered in \cite{amiot}.



\begin{thm} \label{lineorbit} \begin{itemize} \item[(i)] Suppose $g(\Sigma)=0$. Then the action of 
$\mathcal{M}(\Sigma)$ on $G(\Sigma)$ is trivial. Moreover, two line fields $\eta$ and $\theta$ are homotopic if and only if
	\[ w_\eta( \partial_i \Sigma) = w_\theta (\partial_i \Sigma) \ \ \text{\ for all\ } i=1,\ldots d. \] 
\item[(ii)] Suppose $g(\Sigma)=1$. Then two line fields $\eta$ and $\theta$ are in the same 
$\mathcal{M}(\Sigma)$-orbit if and only if 	
\[ w_\eta( \partial_i \Sigma) = w_\theta (\partial_i \Sigma) \ \ \text{\ for all\ } i=1,\ldots d. \] 
	and 
\[ \wt{A}(\eta)=\wt{A}(\theta)\in \Z_{\ge 0},\]
where $\wt{A}(\eta)$ is defined by \eqref{genus-1-line-field-inv}.
\item[(iii)] Suppose $g(\Sigma)\geq 2$. Then two line fields $\eta$ and $\theta$ are in the same 
$\mathcal{M}(\Sigma)$ orbit if and only if the following conditions are satisfied:
\begin{enumerate}
\item $w_\eta( \partial_i \Sigma) = w_\theta (\partial_i \Sigma) \ \ \text{\ for all\ } i=1,\ldots d;$ 
\item $\sigma(\eta)=\sigma(\theta)$ 
(this only needs to be checked if all $w_\eta( \partial_i \Sigma)$ are even);
\item
if $w_\eta(\partial_i \Sigma) = w_\theta (\partial_i \Sigma) \in 2+4 \mathbb{Z}$ and $\sigma(\eta)=\sigma(\theta)=0$
then additionally one must have 
		\[ A(\eta) = A(\theta), \]
		where $A$ is an Arf invariant defined above.
\end{enumerate}

\end{itemize}
\end{thm}

\begin{proof} (i) This follows immediately from the fact that $G(\Sigma)$ is an $H^1(\Sigma)$-torsor and the boundary curves $\partial_i \Sigma$ generate the group $H_1(\Sigma)$.

\noindent
(ii) This is proved in the same way as Theorem 2.8 in \cite{kawazumi}. The main idea is to use the fact that for the standard
choice of simple curves $\a$ and $\b$, the Dehn twists with respect to $\a$ and $\b$ generate an action of 
$\SL_2(\Z)$ on the pair $(w_{\eta}(\a),w_{\eta}(\b))$ (one also uses some other Dehn twists, as in the proof of
\cite[Thm.\ 2.8]{kawazumi}).

\noindent
(iii) We need to prove that if the invariants match then $\eta$ and $\theta$ are in the same $\MM(\Sigma)$-orbit.
Note that $\sigma(\eta)$ is determined by whether the quadratic function $q_\eta$ is trivial modulo $2$ or not.
By Lemma \ref{red-mod4-lem}, it is enough to prove that the quadratic functions $q_\eta$ and $q_\theta$
are in the same $\MM(\Sigma)$-orbit. 

First, let us analyze the result of the action of a transvection
$$T_a(x)=x+(a\cdot x)a$$
on quadratic functions in $\Quad_4$. Note that all such transvections can be realized by elements of the
mapping class group: if the class $a$ is not divisible by $2$ then we can lift it to a primitive element of the homology,
and hence, $T_a$ is realized by some Dehn twist. On the other hand, if $a$ is divisible by $2$ then $T_a=\id$.

We have
\begin{equation}\label{transvection-eq}
q(T_a(x))=q(x)+(a\cdot x)q(a)+2(a\cdot x)(x\cdot a)=q(x)+(q(a)+2)(a\cdot x).
\end{equation}
In particular, if $q(a)=-1$ then $q(T_a(x))=q(x)+ (a\cdot x)$.

Let us set 
$$H:=H_1(\Sigma,\Z_4), \  K=\im(i_*:H_1(\partial\Sigma,\Z_4)\to H_1(\Sigma,\Z_4)).$$

If $q,q'\in\Quad_4$ have $q|_K=q'|_K$ then $(q'-q)$ is a homomorphism $H\to \Z_4$, vanishing on $K$, hence it has form
$x\mapsto (a\cdot x)$ for some $a\in H$.

Assume now that $q\in\Quad_4$ is such that $q|_K$ is surjective, i.e., the reduction of $q|_K$ modulo $2$
is nonzero. Then we claim that any $q'\in\Quad_4$ with $q'|_K=q|_K$ lies in the $\MM(\Sigma)$-orbit of $q$.
Indeed, we have $q'(x)-q(x)=(a\cdot x)$ for some $a\in H$. By surjectivity of $q|_K$ we can find $k\in K$ such that 
$q(k)=-1-q(a)$, i.e., $q(a+k)=-1$. Then from \eqref{transvection-eq} we get
$$q T_{a+k}=q'.$$

Next, let us consider $q\in\Quad_4$ such that $q|_K$ takes values in $2\Z_4$. 
Assume also that $q\mod 2\neq 0$. We claim that 
in this case the $\MM(\Sigma)$-orbit of $q$ is determined by $q|_K$. 
Note that $q\mod 2$ is a homomorphism $H\to \Z_2$ trivial on $K$,
so it is an element of $\Hom(H/K,\Z_2)$. Since $\MM(\Sigma)$ acts transitively on nonzero elements in
$\Hom(H/K,\Z_2)$, it is enough to prove that if $q'\equiv q\mod 2$ and $q'|_K=q|_K$ 
then $q'$ and $q$ are in the same $\MM(\Sigma)$-orbit.
As before we deduce that $q'(x)-q(x)=2(a\cdot x)$ for some $a\in H$. 
If $q(a)\equiv 1\mod 2$ then this immediately gives $q'=qT_a^2$.
On the other hand, if $q'(a)\equiv q(a)\equiv 0\mod 2$ then for any element $b$ with $q(b)\equiv 1\mod 2$
we have
$$qT_{a+b}^2(x)=q(x)+2((a+b)\cdot x)=q'(x)+2(b\cdot x)=q'T_b^2(x),$$
so $q'$ and $q$ are in the same orbit.

Finally, if $q$ takes values in $2\Z_4$ then we have $q=2\cdot q/2$ for a quadratic form $q/2$ on $H/2H$
satisfying \eqref{quadratic-mod2-eq}, and we can use the description of $\MM(\Sigma)$-orbits on such forms
from \cite[Thm.\ 1.3]{kawazumi} (based on the work of Johnson \cite{johnson}).
\end{proof}

\begin{rmk} 
1. It follows from \eqref{chi-winding-eq} that 
the genus of the surface is determined by the boundary winding numbers of $\eta$ via
the formula
\begin{equation}\label{genus-formula}
4-4g(\Sigma)=\sum_{i=1}^d (w_\eta(\partial_i\Sigma)+2).
\end{equation}

\noindent
2. In the case $\sigma(\eta)=0$, the line field $\eta$ is induced by a non-vanishing vector field $v$
(see Lemma \ref{even-wind-numbers-lem}). This induces a spin structure on the surface $\Sigma$ (by considering its mod 2 reduction). The condition that $w_\eta(\partial_i\Sigma)\equiv 2\mod 4$,
for $i=1,\ldots,d$, means that this spin structure extends to a spin structure on the compact surface obtained from $\Sigma$ by capping off the boundaries with a disk. Now, it is a theorem of Atiyah \cite{atiyah} (see also \cite{johnson}) that the action of the mapping class group on the spin structures on a compact Riemann surface has exactly 2 orbits distinguished by the Arf invariant.
\end{rmk}

Theorem \ref{lineorbit} can be used to get a criterion for
a homeomorphism between two {\it different} graded surfaces.

\begin{cor}\label{lineorbit-cor}
Let $(\Si_1,\La_1;\eta_1)$ and $(\Si_2,\La_2;\eta_2)$ be graded surfaces (where $\La_i$ are sets of marked points
on the boundary and $\eta_i$ are line fields) with the same number of boundary components $d$. 
Then there exists an orientation preserving homeomorphism 
$\phi:\Si_1\to\Si_2$ such that $\phi(\La_1)=\La_2$ and $\phi_*(\eta_1)$ is homotopic to $\eta_2$ if and only if
there exists a numbering of boundary components on $\Si_1$ and $\Si_2$ such that
for each $i=1,\ldots,d$, one has
$$\#(\La_1\cap\partial_i\Si_1)=\#(\La_2\cap\partial_i\Si_2),$$
$$w_{\eta_1}(\partial_i\Si_1)=w_{\eta_2}(\partial_i\Si_2),$$
and in addition,
\begin{itemize}
\item if $g(\Si_1)=g(\Si_2)=1$ then $\wt{A}(\eta_1)=\wt{A}(\eta_2)$;
\item if $g(\Si_1)=g(\Si_2)\ge 2$ then $\si(\eta_1)=\si(\eta_2)$ and $A(\eta_1)=A(\eta_2)$ whenever
the latter two invariants are defined.
\end{itemize} 
\end{cor}

\begin{proof} The ``only if" part is clear. For the ``if" part, since $g(\Si_1)=g(\Si_2)$ (due to \eqref{genus-formula}),
we can find a homeomorphism $\wt{\phi}:\Si_1\to\Si_2$ sending $\partial_i\Si_1$ to $\partial_i\Si_2$ and
$\La_1$ to $\La_2$. Applying Theorem \ref{lineorbit} to $\wt{\phi}_*\eta_1$ and $\eta_2$ we deduce the existence
of an element $\psi\in\MM(\Si_2)$ such that $\psi_*(\wt{\phi}_*\eta_1)=\eta_2$.
Thus, the homeomorphism $\phi=\psi \circ \wt{\phi}:\Si_1\to \Si_2$ has the required properties.
\end{proof}

\section{Partially wrapped Fukaya categories}
\label{section2} 

The partially wrapped Fukaya category $\mathcal{W}(\Sigma, \Lambda; \eta)$ (with coefficients in a field $\K$) 
is associated to a 
graded surface $(\Sigma, \Lambda; \eta)$, where $\Sigma$ is a connected
compact surface with non-empty boundary $\partial \Sigma$, $\Lambda \subset \partial \Sigma$ is
 a collection of marked points called \emph{stops}, and  $\eta$ is a line field on $\Sigma$. Partially wrapped Fukaya categories were first introduced in the work of Auroux \cite{auroux} in arbitrary dimension. In the case the symplectic manifold is a surface, which is our focus in this paper, there is a combinatorial description  of $\mathcal{W}(\Sigma, \Lambda; \eta)$ provided in \cite{HKK}. The latter not only gives a topological computation of the partially wrapped Fukaya category defined by symplectic machinery in \cite{auroux}, but also provides an independent, purely topological proof of the invariance of $\mathcal{W}(\Sigma, \Lambda;\eta)$ using the well-known contractibility result of Harer's arc complex. In particular, it follows from this topological description that given two graded surfaces with stops, $(\Sigma_i, \Lambda_i, \eta_i)$ for $i=1,2$, a homeomorphism $\phi : \Sigma_1 \to \Sigma_2$ which restricts to a bijection $\Lambda_1 \to \Lambda_2$ and a homotopy between $\phi_*(\eta_1)$ to $\eta_2$, we get an equivalence between the partially wrapped Fukaya categories $\mathcal{W}(\Sigma_1, \Lambda_1;\eta_1)$ and $\mathcal{W}(\Sigma_2, \Lambda_2;\eta_2)$. The proof of the equivalence of the two approaches, \cite{auroux} and \cite{HKK} given by Abouzaid in the case $\Lambda=\emptyset$ in the appendix of \cite{bocklandt} easily extends to the general case. Another possible approach to this equivalence is via the definition of wrapped Fukaya categories given in \cite{EL} which uses the symplectic field theory formulation. We note that we do not need to appeal to any of these equivalences for the applications in this paper, we simply work with with the definition and the established results given in \cite{HKK}. We next recall this combinatorial description of the partially wrapped Fukaya categories from \cite{HKK}.

A set of pairwise disjoint and
non-isotopic Lagrangians $\{ L_i \}$ in $\Sigma \backslash \Lambda$ generates the
partially wrapped Fukaya category $\mathcal{W}(\Sigma,\Lambda;\eta)$ as a triangulated category
if the complement of the Lagrangians 
\[ \Sigma \setminus \{ \bigsqcup_i L_i \} = \bigcup_f D_f\] 
is a union of disks $D_f$ each of which has at most one stop on its boundary. 
Furthermore, if each $D_f$ has exactly one stop in its boundary, the associative $\K$-algebra \[ A_{L_\bullet}:=\bigoplus_{i,j} \mathrm{hom}(L_i, L_j) \] is formal, and it 
can be described by a graded gentle algebra (see Def. \ref{gentle}). Figure \ref{disk} illustrates how each $D_f$ may look like, where the blue arcs are in $\bigsqcup_i L_i $ while the black arcs lie in $\partial \Sigma$. 

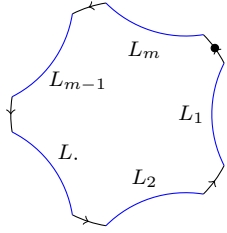
\begin{figure}[htb!]
\centering
\begin{tikzpicture}
\tikzset{vertex/.style = {style=circle,draw, fill,  minimum size = 2pt,inner        sep=1pt}}
\def \radius {1.5cm}
\tikzset{->-/.style={decoration={ markings,
        mark=at position #1 with {\arrow{>}}},postaction={decorate}}}

    \foreach \s in {0, 1,2,3,4 } { 
	    \draw [->-=.5] ({360/5*(\s)+360*3/40}:\radius) arc  ({360/5 *(\s)+360*3/40}:{360/5*(\s+1)-360*3/40
	    }:\radius);
	}
\foreach \s in {0, 1,2,3,4} { 
    \draw[blue] ({360/5 * (\s) - 360*3/40}:\radius) arc ({360/5
    *(\s)+360*3/40}:{360/5*(\s)-360*3/40}:-\radius);

}


\node[vertex] at ({360/5 * 5 + 360*4/40}:\radius) {} ;

    \node at ({360/5}:\radius-0.6cm) {\tiny $L_{m}$};
	\node at ({360/5 *2}:\radius-0.75cm) {\tiny $L_{m-1}$};
    \node at ({360/5 *3}:\radius-0.6cm) {{\tiny $L$}$_\cdot$ };
    \node at ({360/5 *4}:\radius-0.6cm) {\tiny $L_{2}$};
    \node at ({360/5 *5}:\radius-0.6cm) {\tiny $L_1$};

\end{tikzpicture}
    \caption{An example of a disk $D_f$ }
    \label{disk}
\end{figure}

The algebra $A_{L_\bullet}$ can easily be described by a quiver following the flow lines corresponding to rotation around the boundary components of
$\Sigma$ connecting the Lagrangians. Note that each boundary component of $\Sigma$ is an oriented circle (where the boundary orientation is induced by the area form on $\Sigma$). Specifically, a flowline that goes from $L_j$ to $L_i$ gives a generator for $\mathrm{hom}(L_i,L_j)$ (note the reversal of indices). The data of $\Lambda$ enters by disallowing flows that pass through a marked point. 
The algebra structure is given by concatenation of flow lines. Given $\alpha_i \in \mathrm{hom}(L_i,L_{i+1})$ for $i=1,\ldots, n$, we write
\[ \alpha_n \alpha_{n-1} \ldots \alpha_1 \in \mathrm{hom}(L_1,L_{n+1}) \]
for their product, read from right to left, and if non-zero, this expression corresponds to a flow from $L_{n+1}$ to $L_1$. 

Finally, the line field $\eta$ is used to grade the morphism spaces. A convenient way to determine the line field $\eta$ is by describing its restrictions along each of the disks $D_f$. Each such disk is as in Figure \ref{disk}. Different disks are glued along the curves $L_i$ (the blue parts in their boundary). As $L_i$ are contractible, changing a line field by homotopy, we can arrange that it is transverse to $L_i$ everywhere along $L_i$. Every line field on $\Sigma$ (up to homotopy) can be glued out of such line fields on the disks $D_f$.

Note that if we have an embedded segment $c\sub\Sigma$ and a line field $\eta$, which is transversal to $c$ at the ends
$p_1,p_2$ of $c$, then we can define the winding number $w_\eta(c)$ (first, one can trivialize $T\Sigma$ along $c$ in such a way that the tangent line to $c$ is constant, then count the number of times (with sign) $\eta$ coincides with the tangent line to $c$ along $c$. An equivalent definition is given in \cite[Sec. 3.2]{HKK}). Now a line field on a disk $D_f$, transverse to $\{ L_i \}$, is determined (up to homotopy) 
by the integers $\theta_i$, for $i=1,\ldots, m$, given
by its winding numbers along the boundary parts on $\partial \Sigma$ (the black parts in Figure \ref{disk}). 
By definition, these numbers are the degrees of the corresponding morphisms in the wrapped Fukaya category.

The numbers $\theta_i$ can be chosen arbitrarily subject to the constraint
\begin{align} 
\label{constraint} 
\sum_{i=1}^m \theta_i = m-2. \end{align}
This last constraint is the topological condition that needs to be satisfied in order for 
the line field to extend to the interior of the disk. (Note that the stops do not play a role in this discussion.) 

The gentle algebra $A_{L_\bullet}$ is always homologically smooth since so is $\mathcal{W}(\Sigma,\Lambda;\eta)$.
The algebra $A_{L_\bullet}$ is proper (i.e., finite-dimensional) if and only
if there is at least one marked point on every boundary component. The ``if" part is \cite[Cor.\ 3.1]{HKK}. On the other hand, if there is a boundary component with no stops, then we can compose flows along this boundary indefinitely, so $A_{L_\bullet}$ is not proper. 

In what follows, it will be convenient to consider 
$A_{L_{\bullet}}^{op}$ as a quiver algebra $\mathbb{K}Q /I$, so that flow lines from $L_i$ to $L_j$ 
correspond to arrows from the $i^{th}$ vertex to $j^{th}$ vertex. 
Note that the collection $\{L_i\}$ generates the partially wrapped Fukaya category 
$\mathcal{W}(\Sigma, \Lambda; \eta)$. Therefore, we have an equivalence 
\[ D(A_{L_{\bullet}}^{op}) \cong \mathcal{W}(\Sigma, \Lambda; \eta), \]
where the category on the left denotes the bounded derived category of perfect (left) dg-modules over $A_{L_{\bullet}}^{op}$.

\section{Gentle algebras and Fukaya categories}

\subsection{Graded gentle algebras and AAG-invariants}
A quiver is a tuple $Q=(Q_0,Q_1,s,t)$ where $Q_0$ is the set of vertices, $Q_1$ is the set of arrows, $s,t:Q_1\to Q_0$ is the functions that determine the source and target of the arrows. We always assume $Q$ to be finite.
A {\it path} in $Q$ is a sequence of arrows $\alpha_n\ldots \alpha_2\alpha_1$ such that $s(\alpha_{i+1})=t(\alpha_i)$ for $i=1,\ldots, (n-1)$. A {\it cycle} in $Q$ is a path of length
$\ge 1$ in which the beginning and the end vertices coincide but otherwise the vertices are distinct.
For $\K$ a field, let $\K Q$ be the path algebra, with paths in $Q$ as a basis and multiplication induced by concatenation. Note that the source $s$ and target $t$ maps have obvious extensions to paths in $Q$. 

\begin{definition} \label{gentle} A {\it gentle algebra}\footnote{Our terminology is the same as in \cite{Ringel}, 
so we do not impose the condition of finite-dimensionality in the definition of a gentle algebras. What we call ``gentle algebra" 
is sometimes referred to as ``locally gentle algebra".}
 $A = \K Q/I$ is given by a quiver $Q$ with relations $I$ such that 
	\begin{enumerate}
		\item Each vertex has at most two incoming and at most two outgoing edges. 
		\item The ideal $I$ is generated by composable paths of length 2.
		\item For each arrow $\alpha$, there is at most one arrow $\beta$ such that $\alpha \beta \in I$ and there is at most one arrow $\beta$ such that $\beta \alpha \in I$.
		\item For each arrow $\alpha$, there is at most one arrow $\beta$ such that $\alpha \beta \notin I$ and there is at most one arrow $\beta$ such that $\beta \alpha \notin I$.
	\end{enumerate}
In addition, we always assume $Q$ to be connected.	
\end{definition}

We will consider $\Z$-graded gentle algebras, i.e., every arrow in $Q$ should have a degree assigned to it.
For a $\Z$-graded gentle algebra $A$ we denote by $D(A)$ the derived category of perfect dg-modules over $A$,
where $A$ is viewed as a dg-algebra with its natural grading and zero differential. 

\begin{rmk} Note that $D(A)$ is different from the derived category of graded $A$-modules. In fact, the former
category is obtained from the latter as a suitable orbit category (see the discussion in \cite[Sec.\ 1.3]{P-VdB}). 
On the other hand, it is well known that 
if the grading of $A$ is zero then $D(A)$ is equivalent to the perfect derived category of ungraded $A$-modules. 
Indeed, in this case a dg-module over $A$ is the same thing as a complex of $A$-modules.
\end{rmk}

\begin{lem}\label{smooth-gentle-lem} 
(i) A gentle algebra is homologically smooth if and only if there are no \emph{forbidden cycles} i.e. cycles $\alpha_n\ldots \alpha_2\alpha_1$ in $\K Q$ such that $\alpha_{i+1} \alpha_{i} \in I$ for $i \in \Z/n$.

\noindent
(ii) A gentle algebra is proper (i.e., finite-dimensional) if and only if 
there are no \emph{permitted cycles} i.e. paths $\alpha_n \ldots \alpha_2 \alpha_1$ in $\K Q$ such that $\alpha_{i+1} \alpha_{i} \notin I$ for $i \in \Z/n$. 
\end{lem}

\begin{proof} The ``if" direction is proved in \cite[Prop. 3.4(1)]{HKK} using an explicit form of the resolution of the diagonal bimodule. Note that such a resolution goes back to Bardzell's work \cite{Bardzell} (where the case of arbitrary monomial relations is considered). 
It remains to prove that if a gentle algebra $A$ is homologically smooth then there are no forbidden cycles.
Since $A$ is homologically smooth, the diagonal bimodule is a perfect dg-module over $A^{op}\ot A$.
Thus, for every simple $A$-module $S$ (corresponding to one of the vertices), we get a quasi-isomorphism of
$S$ with a perfect dg-module over $A$. It follows that the space $\Ext^*_{A-\dgmod}(S,S)$ is finite-dimensional.
Equivalently, the space $\Ext^*_A(S,S)$, computed in the category of ungraded $A$-modules, is finite-dimensional
(see \cite[Thm.\ 1.3.3]{P-VdB}). But the latter space can be computed using the standard Koszul complex, and
the presence of forbidden cycles would mean that for some $S$ the space $\Ext^*_A(S,S)$ is infinite-dimensional.

\noindent (ii) This is straightforward as properness is equivalent to having only finite number of paths that are nonzero in $A$
(see \cite[Prop. 3.4(2)]{HKK}).
\end{proof}


We will use the following notions from \cite{AAG}.

\begin{definition} A \emph{forbidden path} is a path in $Q$ of the form 
	\[  f= \alpha_{n-1} \ldots \alpha_2 \alpha_1 \in \K Q\]
	such that all $(\alpha_i)$ are distinct and for all $i=1,\ldots, (n-2)$, $\alpha_{i+1}\alpha_i \in I$. It is a \emph{forbidden thread} if for all $\beta \in Q_1$ neither $\beta \alpha_n \ldots \alpha_2 \alpha_1$ nor 
$\alpha_n \ldots \alpha_2 \alpha_1 \beta$ is a forbidden path.
	In addition, if $v \in Q_0$ with $\#\{ \alpha \in Q_1 | s(\alpha)=v\} \leq 1, \#\{ \alpha \in Q_1 | t(\alpha)=v\} \leq 1$, 
	then we consider the idempotent $e_v$ as a (trivial) forbidden thread in the following cases:
	\begin{itemize}
	\item either there are no $\alpha$ with $s(\alpha)=v$ or there are no $\alpha$ with $t(\alpha)=v$;
	\item we have $\beta,\gamma \in Q_1$ with $s(\gamma)=v=t(\beta)$ and $\gamma\beta \in I$.
	\end{itemize}
	The \emph{grading} of a forbidden thread is defined by
	\[ |f|= \sum_{i=1}^{n-1} |\alpha_i| - (n-2).  \]
\end{definition}
\begin{definition} A \emph{permitted path} is a path in $Q$ of the form 
\[  p =\alpha_n \ldots \alpha_2 \alpha_1 \in \K Q\]
	such that all $(\alpha_i)$ are distinct and 
for all $i=1,\ldots, (n-1)$, $\alpha_{i+1}\alpha_i \notin I$, and it is a \emph{permitted thread} if for all $\beta \in Q_1$ neither $\beta \alpha_n \ldots \alpha_2 \alpha_1$ nor 
$\alpha_n \ldots \alpha_2 \alpha_1 \beta$ is a permitted path. 
	In addition, if $v \in Q_0$ with $\#\{ \alpha \in Q_1 | s(\alpha)=v\} \leq 1, \#\{ \alpha \in Q_1 | t(\alpha)=v\} \leq 1$, 
	then we consider the idempotent $e_v$ as a (trivial) permitted thread in the following cases:
	\begin{itemize}
	\item either there are no $\alpha$ with $s(\alpha)=v$ or there are no $\alpha$ with $t(\alpha)=v$;
	\item we have $\beta,\gamma \in Q_1$ with $s(\gamma)=v=t(\beta)$ and $\gamma\beta \notin I$.
	\end{itemize}
	The \emph{grading} of a permitted thread is defined by
	\[ |p|= - \sum_{i=1}^n |\alpha_i|. \]

\end{definition}

\begin{rmk} Inclusion of the idempotents as forbidden and permitted threads ensures that every vertex appears in exactly 
two forbidden threads/cycles and exactly two permitted threads/cycles.
\end{rmk}

\begin{definition} For a gentle algebra $A$, a $\emph{combinatorial boundary component of type I}$ 
is an alternating cyclic sequence of forbidden and permitted threads:
	\[ b= p_nf_n \ldots  p_2 f_2 p_1 f_1 \]
such that $s(f_i) = s(p_i)$ for $i \in \mathbb{Z}/n$, and $t(p_i) =t(f_{i+1})$ for $i \in \mathbb{Z}/n$ with the following condition:

	$(\star)$ For each $i\in \mathbb{Z}/n$, if $f_{i+1} = \alpha_k \ldots \alpha_1$, $p_i = \beta_m \ldots \beta_1$, and $f_{i} = \gamma_n \ldots \gamma_1$ such that $s(f_i)=s(p_i)$ and $t(p_i)=t(f_{i+1})$, we have 
	\[ \gamma_1 \neq \beta_1 \text{  and  } \beta_m \neq \alpha_k. \] 

The winding number associated to a combinatorial boundary component $b$ of type I is defined to be
	\[ w(b):= \sum_{i=1}^r (|p_i| + |f_i|).\]
We also denote the number $n$ of forbidden threads in $b$ as $n(b)$.

A $\emph{combinatorial boundary component of type II}$ (that can appear only if
$A$ is not proper) is simply a permitted cycle
        \[ pc =\alpha_m \ldots \alpha_1.\]
The winding number associated to such a cycle is
        \[ w(pc):= -\sum_{i=1}^m |\alpha_i|.\]

A $\emph{combinatorial boundary component of type II'}$ (that can appear only if
$A$ is not homologically smooth) is simply a forbidden cycle
       \[ fc =\alpha_{m} \ldots \alpha_1.\]
The winding number associated to such a cycle is
        \[ w(fc):= \sum_{i=1}^m |\alpha_i|-m. \]

For combinatorial boundary components of types II and II' we set $n(b)=0$.
\end{definition}

\begin{lem}\label{AAG-lem}
Let $A$ be a proper gentle algebra, with grading in degree zero.
Then the collection of pairs $(n(b),n(b)-w(b))$, over all
	combinatorial boundary components (taken with multiplicities)
	coincides with AAG-invariants of $A$. 
\end{lem}

\begin{proof} This follows directly from the description of the AAG-invariants in \cite[Sec. 3]{AAG}.
Note that the pair $(0,m)$ in Step (3) of
                the algorithm of \cite[Sec. 3]{AAG} associated to a forbidden cycle $fc=\alpha_m\ldots\alpha_1$ matches with the pair
                $(0,w(fc))$
                associated with the corresponding combinatorial component of type II'. Indeed, $w(fc)=m$ since the grading of $A$ is in degree 0.
\end{proof}

{\it From now on we will always assume that our gentle algebras are homologically smooth},
with the exception of Remark \ref{non-smooth-rem}.

Motivated by Lemma \ref{AAG-lem} we extend the definition of the AAG-invariants to graded gentle algebras.

\begin{definition} For a graded gentle algebra $A$ we define the {\it AAG-invariants} to be the collection of pairs
$(n(b),n(b)-w(b))$, taken with multiplicities, where $b$ runs over all combinatorial boundary components of $A$.
\end{definition}

\subsection{Relation to Fukaya categories}

The definition of the combinatorial boundary component for a gentle algebra is motivated by the following proposition:

\begin{prop} \label{combinatorial} Suppose $\Sigma$ is a surface with a collection of marked points 
$\Lambda \subset \partial \Sigma$, and a line field $\eta$. 
Let $\{L_i\}$ be a collection of Lagrangians such that the complement of $\bigsqcup_{i} L_i$ is a union of 
disks each of which has exactly one stop on its boundary. 
Then the combinatorial boundary components of the homologically smooth gentle algebra 
$A = \bigl(\bigoplus_{i,j} \mathrm{hom}(L_i,L_j)\bigr)^{op}$ 
are in natural bijection with the boundary components of $\partial \Sigma$. 
Furthermore, if a combinatorial boundary component $b$ corresponds to a boundary component $B\sub \partial \Sigma$
then the number of forbidden threads in $b$ is equal to the number of stops on $B$ and the winding numbers match:
$$w_\eta(B)=w(b).$$
\end{prop}

\begin{proof} Figure \ref{boundary} shows an example of the way the surface $\Sigma$ looks around a boundary component $B$. Assume first that there is at least one stop on $B$. Let 
$$q_1(1),\ldots,q_1(k_1),q_2(1),\ldots,q_2(k_2),\ldots,q_n(1),\ldots,q_n(k_n)$$
be the endpoints of the Lagrangians ending on $B$, ordered compatibly with the orientation of $B$. Here we assume
that there are no stops between $q_i(j)$ and $q_i(j+1)$ and there is exactly one stop $s_i$ between
$q_i(k_i)$ and $q_{i+1}(1)$, for $i\in \Z/n$. Then for every $i\in\Z/n$ we have a permitted thread $p_i=\b_i(k_i-1)\ldots\b_i(1)$,
where $\b_i(j)$ is the generator of $A$ corresponding to the flow on $B$ from $q_i(j)$ to $q_i(j+1)$. On the other hand,
each stop $s_i$ lies on a unique disk $D$, and by looking at the pieces of $\partial D$ formed by other boundary components
of $\Sigma$, we obtain a forbidden thread $f_i= \alpha_{m_i}\ldots \alpha_1$ starting at the Lagrangian corresponding to
$q_i(1)$ and ending at the one corresponding to $q_{i-1}(k_{i-1})$. Thus, we get a combinatorial boundary component of
type I, $b=p_nf_n\ldots p_1f_1$.

The winding number of $\eta$ along the arc passing through the stop, oriented in the opposite direction to the boundary direction, is determined using the constraint \eqref{constraint} to be
	\[ |f| = \sum_{i=1}^{n-1} |\alpha_i| - (n-2) \]
	On the other hand, the winding number of $\eta$ along the arc corresponding to the permitted thread $p$ is simply $|p|$. 
Thus, we get the equality $w_\eta(B)=w(b)$.

In the case of a boundary component $B\sub\partial\Sigma$ with no stops, the sequence of flows between
the corresponding ends of Lagrangians on $B$ gives a permitted cycle, i.e., a combinatorial boundary component of type II.
Again, the winding numbers match.

It is easy to see that in this way we get a bijection between the boundary components $B$ and the combinatorial boundary
components of $A$.	

\begin{figure}[!h]
\begin{tikzpicture}
\tikzset{vertex/.style = {style=circle,draw, fill,  minimum size = 2pt,inner        sep=1pt}}
\def \radius {2.5cm}
\tikzset{->-/.style={decoration={ markings,
        mark=at position #1 with {\arrow{<}}},postaction={decorate}}}

\draw[->-=.7] ({0}:\radius/5) arc ({0}:{360/5*5}:\radius/5);
\draw[->-=.3] ({360/5 *5}:\radius/7)+(0.36,2.22) arc ({360/5 *2}:{360/5*5}:\radius/7);
\draw[->-=.2] ({360/5 *5}:\radius/7)+(-2,2.7) arc ({360/5 *3}:{360/5*5}:\radius/7);
\draw[->-=.5] ({360/5 *5}:\radius/7)+(-4,2) arc ({360/5 *3}:{360/5*5.5}:\radius/7);
\draw[->-=.5] ({360/5 *5}:\radius/7)+(-3.8,-0.5) arc ({360/5 *3.7}:{360/5*6.3}:\radius/7);
\draw[->-=.5] ({360/5 *5}:\radius/7)+(-2,-2.7) arc ({360/5 *5}:{360/5*7}:\radius/7);
\draw[->-=.5] ({360/5 *5}:\radius/7)+(1.3,-2) arc ({360/5 *1}:{360/5*3}:\radius/7);
\draw[->-=.5] ({360/5 *5}:\radius/7)+(2.1,0.3) arc ({360/5 *1}:{360/5*3.5}:\radius/7);

\node[vertex] at ({360/5 * 4.6}:\radius/5) {} ;
\node[vertex] at ({360/5 *2}:\radius/5) {} ;
\node[vertex] at (1.1,1.67) {};
\node[vertex] at (-1.1,2.65) {};
\node[vertex] at (1.2,-2.3) {};

	\draw[blue] (0.3,0.4) to (0.84,1.7);
\draw [blue] (-0.2,0.45) to (-1.4,2.55);
\draw [blue] (-0.4, -0.3) to (-3.1,-0.3);
\draw [blue] (0.3, -0.4) to (1.4,-2);
\draw [blue] (0.5,0) to (2,0);

\draw[blue] (-1.6,2.65) to (-3,2.2);
\draw[blue] (-3.4,1.85) to (-3.4,0.21);
\draw[blue] (-3.35,-0.5) to (-2.2,-2.4);
\draw[blue] (-1.7,-2.5) to (1.24,-2.5);
\draw[blue] (1.55,-1.97) to (2.1,-0.3); 

\draw[blue] (2.2,0.28) to (1.3,1.8); 
\draw[blue] (0.68,2.15) to (-1,2.85);

\node at (-1.5, 2.3) {\tiny $\alpha_1$};
\node at (-2.9, 1.8) {\tiny $\alpha_2$};
\node at (-2.9, 0.2) {\tiny $\alpha_3$};

\node at (0.1, 0.7) {\tiny $\beta_1$};
\node at (0.7, 0.2) {\tiny $\beta_2$};

\node at (1.45, -1.8) {\tiny $\gamma_1$};
\node at (1.85, -0.3) {\tiny $\gamma_2$};

\node at (-0.2, -0.7) {\tiny $\delta_1$};

\node at (0.45, 1.8) {\tiny $\tilde{\beta}_2$};

\end{tikzpicture}
	\caption{The boundary component is given by the cyclic sequence $p_2f_2p_1f_1$ where $f_1= \alpha_3 \alpha_2 \alpha_1$, $p_1= \beta_2 \beta_1$, $f_2 = \gamma_2 \gamma_1$ and $p_2 =\delta_1$. Note that if instead of $f_1$, we considered the forbidden thread $\tilde{f}_1 = \tilde{\beta}_2 \beta_1$, the condition $(\star)$ is violated. }
\label{boundary} 	
\end{figure}
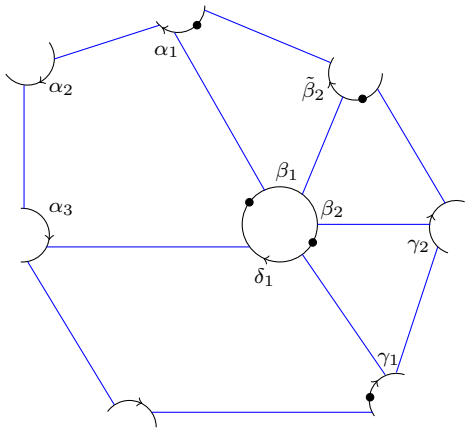

\end{proof}

Let $A$ be a homologically smooth gentle algebra. 
We associate with $A$ a ribbon graph $\mathcal{R}_A$ 
whose vertices are in bijection with the collection of forbidden threads in $Q$, and whose edges are in bijection with vertices of $Q$. More precisely, recall that there are precisely two forbidden threads that pass through a vertex of $Q$. The corresponding edge on $\mathcal{R}_A$ is defined to connect the two forbidden threads. 

Next, we will define a ribbon structure, i.e., a cyclic order on the set of edges incident to each vertex. 
In fact, we will equip each such set of edges with a total order which will induce a cyclic order.
	(Thus, we get what is called a \emph{ciliated fat graph} \cite{fockrosly}.)
Namely, the set of edges incident to a vertex $f$ of $\mathcal{R}$ is in bijection with the set of vertices of $Q$ which appear in the forbidden thread $f$. Now we use the order in which these vertices appear in the forbidden thread $f$. 

Thus, we can consider the associated thickened surface 
$\Sigma_A$ such that $\mathcal{R}_A$ is embedded as a deformation retract of $\Sigma_A$. 
More specifically, to construct $\Sigma_A$ we replace each vertex of $\mathcal{R}_A$ with a 2-disk $\mathbb{D}^2$ and each edge with a strip, a thin oriented rectangle $[-\epsilon, \epsilon] \times [0,1]$, where the rectangles are attached to the boundary of the disks according to the given cyclic orders at the vertices. On the boundary of each disk associated to the vertex of $\mathcal{R}_A$ we
	also mark a point, called a \emph{stop} as follows. If the linear order
	on edges incident to this vertex is given by   $e_1 < e_2 <  \ldots <
	e_k$, the stop $e_0$ appears in the circular order such that $e_k < e_0
	< e_1$.  We define $\Lambda_A$ by taking the union of all such points. In particular, the cardinality of $\Lambda_A$, is equal to the number of forbidden threads in $A$.

\begin{thm}\label{gentle-Fuk-thm} (i) Given a homologically smooth gentle algebra $A$ over a field $\mathbb{K}$ 
(with $|Q_1|>0$), let $(\Sigma_A, \Lambda_A)$ be the corresponding surface with stops defined above. 
Then $\Sigma_A$ is connected with non-empty boundary, and for each $\Z$-grading on $A$ there is a natural
line field $\eta$ on $\Sigma$ such that we have a derived equivalence
	\[ D(A) \cong \mathcal{W}(\Sigma_A, \Lambda_A; \eta_A). \]
Furthermore, the AAG-invariants of $A$ are given by the collection
          of pairs 
          $$(n_i,n_i - w_{\eta_A}(\partial_i \Sigma_A)),$$ 
          where $(\partial_i\Sigma_A)_{i=1,\ldots,N})$ are all boundary components of $\Sigma_A$ and
	  $n_i \in \mathbb{Z}_{\geq0}$ is the number of marked points on $\partial_i \Sigma_A$.

\noindent
(ii) One has 
$$\chi(\Sigma_A)=\chi(Q)=|Q_0|-|Q_1|.$$ 
\end{thm}

\begin{proof} (i) 
First, let us check that the ribbon graph $\mathcal{R}_A$ and hence the associated surface $\Sigma_A$ is connected.
Indeed, for every vertex $v$ of $Q$ let $e(v)$ be the corresponding edge in $\mathcal{R}_A$, viewed as a subgraph in
$\mathcal{R}_A$. Since $Q$ is connected, it is enough to check that if $v$ and $v'$ are connected by an edge $\alpha$ in 
$Q$ then $e(v)$ and $e(v')$ intersect in $\mathcal{R}_A$. Indeed, let $f$ be a forbidden thread containing $\alpha$
(it always exists). Then $f$ is a vertex of both $e(v)$ and $e(v')$. This proves our claim that $\mathcal{R}_A$ is
connected.

Dual to the edges of $\mathcal{R}_A$ we obtain a disjoint collection of non-compact arcs $L_v$ indexed by vertices of $Q$. Thus, $\Sigma_A$ is a surface with non-empty oriented boundary, 
$\Lambda_A$ is a set of marked points in its boundary, 
and $\{L_v : v \in Q_0\}$ is a collection pair-wise disjoint and non-isotopic Lagrangian arcs in 
$\Sigma_A \setminus \Lambda_A$. Furthermore, the complement
	\[ \Sigma_A \setminus \{ \bigsqcup_v L_v \} = \bigcup_f D_f\] 
	is a union of disks $D_f$ indexed by forbidden threads $f$ in $Q$, with exactly one marked point on its boundary (see Examples \ref{ex1}, \ref{ex2} below). In particular, the collection $\{L_v\}$ gives a generating set.

	By construction, there is a bijection between arrows in the quiver $Q$ and 
  the generators of the endomorphism algebra $A_L:=\bigoplus_{v,w} \mathrm{hom}(L_v,L_w)$ 
 since each edge $\alpha$ in $Q$ is in exactly one forbidden thread $f$, 
 and the corresponding $D_f$ has a flow associated to $\alpha$. Furthermore, two flows 
  $\alpha_1 : L_{v_2} \to L_{v_1}$ and $\alpha_2: L_{v_3} \to L_{v_2}$ 
   can be composed in $A_L$ if and only if $\alpha_i$ is in a forbidden thread $f_i$, for $i=1,2$, 
  such that the disks $D_{f_1}$ and $D_{f_2}$ are glued along the edge corresponding to $v_2$. 
  But this means that the corresponding elements of $A$ satisfy $\alpha_2 \alpha_1 \notin I$, 
  as otherwise condition (3) of Definition \ref{gentle} would be violated. This implies that 
   $A$ is naturally identified with $A_L^{op}$ as an ungraded algebra.

	We define the line field $\eta_A$ on $\Sigma_A$ as follows. 
      We require that the line field is transverse to each $L_v$. 
    Then it suffices to describe its restrictions to the disks $D_f$. 
    Each $D_f$ is a $2m$-gon as in Figure \ref{disk}. As explained in Section \ref{section2}, the homotopy class of a line field on $D_f$ is determined by the winding numbers
	$\theta_i$ along the boundary arcs of $D_f$, $\alpha_i$, for $i=1,\ldots,(m-1)$, 
   avoiding the unique stop (black in Figure \ref{disk}) between the Lagrangians (blue in Figure \ref{disk}). 
         Indeed, the remaining winding number $\theta_m$
along the boundary arc that passes through the stop is determined by the condition  $\sum_{i=1}^m \theta_i = m-2$,
and we can define $\eta_A|_{D_f}$ as the unique line field with these winding numbers.
Now we set
$\theta_i$, for $i=1,\ldots,m-1$, to be the degree of the generator of $A$ corresponding to $\alpha_i$. 
	
With this definition
	$A$ and $A_L^{op}$ are identified as graded algebras. 
Since we also know that the collection $\{L_v\}$ generates $\mathcal{W}(\Sigma_A, \Lambda_A; \eta_A)$, 
         we conclude that 
	\[ D(A) \cong \mathcal{W}(\Sigma_A, \Lambda_A;\eta_A). \]

	Finally, the last statement follows from Proposition \ref{combinatorial}.

\noindent
(ii) We have $\chi(\Sigma_A)=\chi(\mathcal{R}_A)$. Let us denote by $v(\RR_A)$ and $e(\RR_A)$
the numbers of vertices and edges in $\RR_A$. We have $e(\RR_1)=|Q_0|$, while $v(\RR_A)$ is the number of forbidden
threads. Let $f_1,\ldots,f_m$ be all forbidden threads. Since every edge of $Q$ belongs to the unique forbidden thread,
we have 
$$\sum \ell(f_i)=|Q_1|$$
(where $\ell(\cdot)$ is the length). On the other hand, since every vertex is contained in exactly two forbidden threads, we have
$$\sum (\ell(f_i)+1)=2|Q_0|.$$
Combining this with the previous formula we get 
$$v(\RR_A)=2|Q_0|-|Q_1|,$$
so we deduce that $\chi(\RR_A)=\chi(Q)$.
\end{proof}

Using formula \eqref{genus-formula} we derive the following property of the AAG-invariants.

\begin{cor}\label{genus-cor} 
Let $\{(n_i,m_i)\}_{i=1,\ldots,d}$ be the AAG-invariants of a homologically smooth graded gentle algebra $A$.
Then
$$\sum_{i=1}^d (n_i-m_i+2)=4-4g,$$
where $g\ge 0$ is the genus of the corresponding surface $\Sigma_A$.
\end{cor}

Combining Theorem \ref{gentle-Fuk-thm} with Corollary \ref{lineorbit-cor}, we get the following result.

\begin{cor}\label{gentle-Fuk-cor} 
Given two homologically smooth graded gentle algebras $A$ and $B$, assume 
that the AAG-invariants of $A$ and $B$ are the same, and in addition, the invariants $\wt{A}(\cdot)$, $\si(\cdot)$ and $A(\cdot)$ 
(see Theorem \ref{lineorbit})
of the line fields $\eta_A$ (on $\Sigma_A$) and $\eta_B$ (on $\Sigma_B)$ are the same whenever they are defined.
      Then $D(A)\simeq D(B)$.
\end{cor}

As a particular case of the last Corollary, we can describe some cases when already looking at the AAG-invariants 
 gives the derived equivalence.     
      
\begin{cor}\label{AAG-cor} Assume that $A$ and $B$ are homologically smooth graded gentle algebras, such that
the AAG-invariants of $A$ and $B$ coincide (up to permutation) and are given by a collection
 $\{(n_i,m_i)\}_{i=1,\ldots,d}$. Assume in addition that one of the following conditions holds:
  
 \noindent
 (a) $\sum_i (n_i-m_i+2)=4$;
 
 \noindent
 (b) $\sum_i (n_i-m_i+2)=0$ and $\gcd(n_1-m_1+2,\ldots,n_d-m_d+2)=1$;
 
 \noindent
 (c) $\sum_i (n_i-m_i+2)<0$ and at least one of the numbers $n_i-m_i$ is odd.
 
      Then $D(A)\simeq D(B)$.
\end{cor}

\begin{proof} By Corollary \ref{genus-cor}, the three cases are distinguished by the genus $g(\Sigma_A)$:
in case (a) it is $0$, in case (b) it is $1$, and in case (c) it is $>1$. Now the assertion follows from 
Corollary \ref{lineorbit-cor}.
\end{proof}
         
\begin{rmk}
There is a simple combinatorial recipe for calculating winding numbers of the line field $\eta$ on $\Sigma_A$,
along the loops corresponding to cycles in the graph $\RR_A$. Note that knowing these numbers is
enough to calculate all the invariants of $\eta$.
                    
Indeed, a cycle in the graph $\RR_A$ is an alternating sequence $\ldots v_if_iv_{i+1}f_{i+1}\ldots$
of vertices and forbidden threads in $Q$. Since $\eta$ is transverse to each arc $L_{v_i}$, we can calculate the
winding number as the sum of winding numbers of the segments of the cycle connecting a point in $L_{v_i}$ with
a point in $L_{v_{i+1}}$ through the disk $D_{f_i}$. 
Recall that the boundary of this disk is formed by the arcs $L_v$ where $v$ runs through vertices in the thread $f_i$,
and the parts of the boundary labeled by arrows in $f_i$ (as in Figure \ref{boundary}). Now we claim that 
the contribution to the winding number from the segment $v_if_iv_{i+1}$ is equal to
$$w_\eta(v_if_iv_{i+1})=\begin{cases} 1-m+\sum_{j=1}^m \deg(\a_j), & v_i \xrightarrow{\a_1}\ldots\xrightarrow{\a_m}v_{i+1}\sub f_i\\ 
-1+m-\sum_{j=1}^m \deg(\b_j), & v_{i+1}\xrightarrow{\b_1}\ldots\xrightarrow{\b_m}v_i\sub f_i.\end{cases}$$
Indeed, this follows immediately by looking at the polygon formed by the arcs $L_{v_i}$ and $L_{v_{i+1}}$, by
the segment of our cycle between them, and by the part of the boundary of $D_{f_i}$ between these arcs.
\end{rmk}                    
                    
\subsection{Application to finite-dimensional gentle algebras and examples}                    

It is well known that gentle algebras are Koszul and that the Koszul dual of a gentle algebra is again gentle,
corresponding to the dual combinatorial data (see \cite[Sec.\ 3.3]{BH} where what we call ``gentle" is called
``locally gentle"). Furthermore, under this duality homologically smooth graded gentle algebras are exchanged
with finite-dimensional ones.
Thus, using Koszul duality we can convert our results into those
about finite-dimensional gentle algebras. 
      
Let $A$ be a finite-dimensional gentle algebra with grading in degree $0$. Let $A^!$ be the Koszul dual
gentle algebra (with respect to the generators given by the edges). We equip $A^!$ with the grading for which all
edges have degree $1$ (i.e., path-length grading).
Then the result of Keller in \cite[Sec.\ 10.5]{keller} (``exterior" case) gives an equivalence 
$$D_f(A)\simeq D(A^!),$$                                                                
where $D_f(A)$ is the bounded derived category of finite-dimensional $A$-modules (and $D(A^!)$ is the perfect derived
category of $A^!$ viewed as a dg-algebra, as before). 

 Furthermore, it is easy to check that the AAG-invariants of $A$ and $A^!$ are the same.                                                                
Thus, Corollary \ref{AAG-cor} leads to the following result.                         \begin{cor}\label{AAG-fd-cor} Let $A$ and $B$ be finite-dimensional gentle algebras with grading in degree $0$, such that
the AAG-invariants of $A$ and $B$ coincide (up to permutation) and satisfy one of the conditions (a)--(c) of 
 Corollary \ref{AAG-cor}. Then                                                                  $$D_f(A)\simeq D_f(B).$$                                                  
\end{cor}
                                                                 
\begin{example}
\label{ex1} 
Here is an example illustrating the construction of associating a surface to a gentle algebra. Consider the gentle algebra given in Figure \ref{anexampleof}.
\begin{figure}[H]
\begin{tikzpicture}
    \tikzset{vertex/.style = {style=circle,draw, fill,  minimum size = 2pt,inner    sep=1pt}}

	\tikzset{edge/.style = {->,-stealth',shorten >=8pt, shorten <=8pt  }}

\node[vertex] (a) at  (0,0) {};
\node[vertex] (a1) at (1.5,0) {};
\node[vertex] (a2) at (3,0) {};
\node[vertex] (a3) at (4.5,0) {};

\node at  (0,0.3) {\tiny 1};
\node at (1.5,0.3) {\tiny 2};
\node at (3,0.3) {\tiny 3};
\node at (4.5,0.3) {\tiny 4};


\draw[edge] (a)  to[in=150,out=30] (a1);
\draw[edge] (a)  to[in=210,out=330] (a1);
\draw[edge] (a1) to (a2);
\draw[edge] (a2) to (a3);

\node at (0.75,-0.5) {\tiny $d$};
\node at (0.75,0.5) {\tiny $a$};
\node at (2.25,-0.3) {\tiny $b$};
\node at (3.75,-0.3) {\tiny $c$};

\node at (7,0) {\small $bd=0$};

\end{tikzpicture}
\caption{An example of gentle algebra}
\label{anexampleof}
\end{figure}
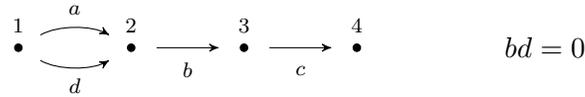

The forbidden threads are given by  $\{ a, bd,c, e_4\}$. The permitted threads are given by $\{ cba, d, e_3, e_4 \}$. The combinatorial boundary components are given by
	$\{ p_3 f_3 p_2 f_2 p_1 f_1, p_4 f_4 \}$ where, $f_1=e_4$, $p_1=e_4$, $f_2=c$, $p_2=e_3$, $f_3=bd$, $p_3=cba$, and $f_4 = a$, $p_4=d$.

The associated ribbon graph is given in Figure \ref{fig2}, 
where the cyclic order at vertices are given by counter-clockwise rotation.
\begin{figure}[H]
\begin{tikzpicture}
    \tikzset{vertex/.style = {style=circle,draw,  minimum size = 20pt,inner    sep=1pt}}
 \tikzset{stop/.style = {style=circle,draw, fill,  minimum size = 2pt,inner    sep=1pt}}

	\tikzset{edge/.style = {}}

\node[vertex] (a) at  (0,0) {};
\node[vertex] (a1) at (1.5,0) {};
\node[vertex] (a2) at (3,0) {};
\node[vertex] (a3) at (4.5,0) {};

\node[stop] (b) at (0.35,0) {};
\node[stop] (b1) at (4.85,0) {};
\node[stop] (b2) at (1.75,0.25) {};
\node[stop] (b2) at (3,0.35) {};

	\node at (0,0) {\tiny $a$};
	\node at(1.5,0) {\tiny $bd$};
	\node at (3,0) {\tiny $c$};
	\node at(4.5,0) {\tiny $e_4$};

	\node at (0.75,-0.8) {\tiny $2$};
	\node at (0.75, 0.8) {\tiny $1$};
	\node at (3.75,0.2) {\tiny $4$};
	\node at (2.25,0.2) {\tiny $3$};


\draw[edge] (a)  to[in=120,out=60] (a1);
\draw[edge] (a)  to[in=240,out=300] (a1);
\draw[edge] (a1) to (a2);
\draw[edge] (a2) to (a3);

\end{tikzpicture}
	\caption{Ribbon graph associated to a gentle algebra}
\label{fig2}
	\end{figure}
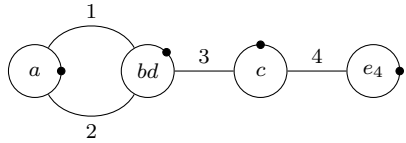
Figure \ref{surf} depicts the corresponding surface, together with the dual arcs $L_1,L_2,L_3,L_4$.

As this is a genus zero surface, the line field is determined by the winding numbers along the boundary components. The winding number along the interior puncture which corresponds to the combinatorial boundary component $p_4f_4$ is given by $|a|-|d|$ and the winding number along the outer boundary component which corresponds to the combinatorial boundary component $p_3f_3p_2f_2p_1f_1$ is the negative of this (since the two boundary components are homotopic but oriented in an opposite way) but can also be computed as $ (-|a|-|b|-|c|)+(|b|+|d|-1)+|c|+1 = |d|-|a|$.
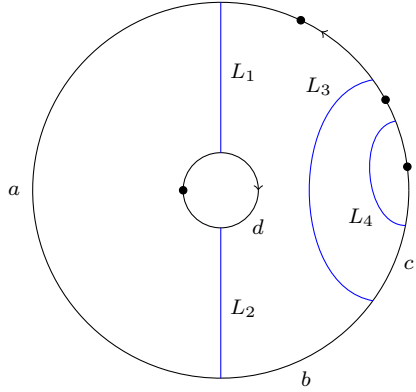
\begin{figure}[H]
\centering
\begin{tikzpicture}
\tikzset{vertex/.style = {style=circle,draw, fill,  minimum size = 2pt,inner        sep=1pt}}
\def \radius {2.5cm}
\tikzset{->-/.style={decoration={ markings,
        mark=at position #1 with {\arrow{>}}},postaction={decorate}}}

\draw ({360/5}:\radius) arc (360/5:360/5 *5:\radius);

\draw[->-=.8] ({0}:\radius) arc ({360/5 *0}:{360/5}:\radius);

\draw[->-=.0] ({360/5 *5}:\radius/5) arc ({360/5 *5}:0:\radius/5);

	\draw[blue] (0,\radius/5) to (0,\radius);
	\draw[blue] (0,-\radius/5) to (0,-\radius);

	\begin{scope}[xscale=-1]

		\draw[blue] ({360/5 * 2}:\radius) to[in=10,out=350] ({360/5 * 3}:\radius);
		\draw[blue] ({360/5 * 2.2}:\radius) to[in=0,out=350] ({360/5 * 2.65}:\radius);

 \node[vertex] at ({360/5 * 1.6}:\radius) {} ;
 \node[vertex] at ({360/5 * 2.1}:\radius) {} ;
 \node[vertex] at ({360/5 * 2.4}:\radius) {} ;
     \node at ({360/5 *1.85}:\radius-0.6cm) {\tiny $L_3$};
    \node at ({360/5 *2.65}:\radius-0.6cm) {\tiny $L_4$ };

\node[vertex] at ({360/5 * 5}:\radius/5) {} ;

    \node[xshift=-7] at ({0}:\radius) {\tiny $a$} ;
   \node[yshift=-7, xshift=2] at ({360/5 *3.4 }:\radius)  {\tiny $b$} ;

    \node[xshift=5,yshift=-2] at ({360/5 * (2.8)}:\radius)  {\tiny $c$} ;
    \node[xshift=5, yshift=-3] at ({360/5 * (3.2)}:\radius/5)  {\tiny $d$} ;

	\end{scope}

    \node at ({360/5*1.1}:\radius-0.9cm) {\tiny $L_{1}$};
    \node at ({360/5*3.9}:\radius-0.9cm) {\tiny $L_{2}$};

\end{tikzpicture}
    \caption{Surface associated to a gentle algebra}
    \label{surf}
\end{figure}

\end{example}

\begin{example} \label{ex2} Here is another example that produces a genus 1 surface with 2 boundary components. Consider the gentle algebra given by Figure \ref{anothergentle}.

\begin{figure}[H]
\centering
\begin{tikzpicture}
    \tikzset{vertex/.style = {style=circle,draw, fill,  minimum size = 2pt,inner    sep=1pt}}

	\tikzset{edge/.style = {->,-stealth',shorten >=8pt, shorten <=8pt  }}

\node[vertex] (a) at  (0,0) {};
\node[vertex] (a1) at (1.5,0) {};
\node[vertex] (a2) at (3,0) {};
\node[vertex] (b) at (0,1) {};
\node[vertex] (b1) at (1.5,1) {};
\node[vertex] (b2) at (3,1) {};
\

\node at  (0,1.3) {\tiny 1};
\node at (1.5,1.3) {\tiny 2};
\node at (3,1.3) {\tiny 3};
\node at  (0,-0.3) {\tiny 4};
\node at (1.5,-0.3) {\tiny 5};
\node at (3,-0.3) {\tiny 6};


\draw[edge] (a) to (a1);
\draw[edge] (a1) to (a2);
\draw[edge] (b) to (b1);
\draw[edge] (b1) to (b2);
\draw[edge] (a) to (b1);
\draw[edge] (b) to (a1);
\draw[edge] (a1) to (b2);
\draw[edge] (b1) to (a2);

\node at (0.75,1.2) {\tiny $a$};
\node at (2.25,1.2) {\tiny $b$};
\node at (1.3,0.3) {\tiny $c$};
\node at (2.8,0.7) {\tiny $d$};

\node at (0.75,-0.2) {\tiny $t$};
\node at (2.25,-0.2) {\tiny $x$};
\node at (1.35,0.7) {\tiny $y$};
\node at (2.8,0.3) {\tiny $z$};

\node at (7,0.5) {\small $za=by=xc=dt=0$};

\end{tikzpicture}
\caption{Another example of a gentle algebra}
\label{anothergentle}
\end{figure}
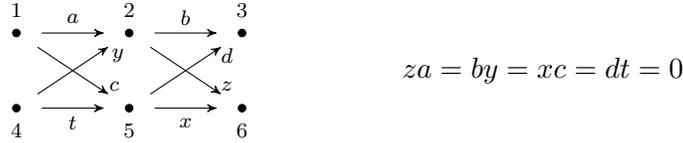

The forbidden threads are given by $\{za,by,xc,dt\}$, and the permitted threads are given by $\{ba,dc,xt,zy\}$. The combinatorial boundary components are given by $\{p_2f_2p_1f_1, p_4f_4p_3f_3\}$ where $f_1=dt, p_1=zy, f_2=xc, p_2=ba$, and $f_3=za,p_3=dc, f_4=by,p_4=xt$. 

The corresponding surface is given in Figure \ref{gen1}.

This is a genus 1 surface with 2 boundary components. To determine the line field we need to compute its winding number along the booundary components corresponding to $b_1= p_2f_2p_1f_1$ and $b_2= p_4f_4p_3f_3$ as well as winding numbers along non-separating curves depicted in grey. The horizontal one corresponds to the cycle $\alpha= f_2v_5f_1v_4f_4v_2f_3v_1$, and the vertical one corresponds to the cycle $\beta = f_1v_5f_2v_1f_3v_2f_4v_3$. From the formulae given, it is easy to compute 
\begin{align*} 
w_\eta( b_1 ) &= -|a|-|b|+(|x|+|c|-1)-|z|-|y|+(|d|+|t|-1) \\
w_\eta( b_2 ) &= -|t|-|x|+(|y|+|b|-1)-|c|-|d|+(|z|+|a|-1) \\
w_\eta(\alpha) & =  |t|-|y|+|a|-|c|\\
w_\eta(\beta) & =  -|b|-|a|+|c|+|d|
\end{align*} 
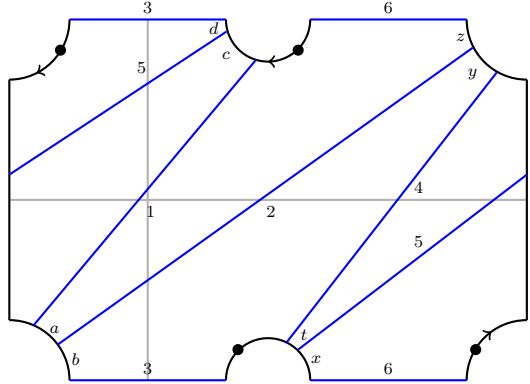
\begin{figure}[ht]
\centering
    \begin{tikzpicture}[thick,scale=0.8, every node/.style={transform shape}]
\tikzset{vertex/.style = {style=circle,draw, fill,  minimum size = 2pt,inner        sep=1.5pt}}
    \tikzset{->-/.style={decoration={ markings,
        mark=at position #1 with {\arrow{>}}},postaction={decorate}}}

    \draw [thick=1.5] (1,0) arc (0:90:1);
    \draw [->-=.7, thick=1.5] (1,6) arc (360:270:1);
    \draw [thick=1.5](7.6,6) arc (180:270:1);
    \draw [->-=.6, thick=1.5](7.6,0) arc (180:90:1);

\draw[black!30] (0,3) -- (8.6,3);
\draw[black!30] (2.3,0) -- (2.3,6);

\draw  [thick=1.5]((5,0) arc (0:180:0.7);

\draw  [->-=.5, thick=1.5]((5,6) arc (0:-180:0.7);

\draw [thick=1.5] (0,1) -- (0,5);
\draw [thick=1.5] (8.6,1) -- (8.6,5);

\draw [blue] (1,0) -- (3.6,0);
\node at (2.3,0.2) {\tiny $3$};
\draw [blue] (5,0) -- (7.6,0);
\node at (6.3,0.2) {\tiny $6$};

\draw [blue] (1,6) -- (3.6,6);
\node at (2.3,6.2) {\tiny $3$};
 \draw [blue] (5,6) -- (7.6,6);
\node at (6.3,6.2) {\tiny $6$};

\draw[blue] (0.4,0.91) -- (4.1,5.33);
\node at (2.35, 2.8) {\tiny $1$};
   
\draw[blue] (0.8,0.59) --  (7.7,5.53);
\node at (4.35, 2.8) {\tiny $2$};
	
\draw[blue] (4.6,0.62) --  (8.1,5.12);
\node at (6.8, 3.2) {\tiny $4$};
	    
\draw[blue] (0,3.42) -- (3.61,5.8);
\node at (2.2, 5.2) {\tiny $5$};

\draw[blue] (4.8,0.51) -- (8.6,3.42);
\node at (6.8,2.3) {\tiny $5$};

\node at (0.75,0.85) {\tiny $a$};
\node at (1.1,0.35) {\tiny $b$};
\node at (3.6,5.4) {\tiny $c$};
\node at (3.4,5.85) {\tiny $d$};
\node at (4.9,0.75) {\tiny $t$};
\node at (5.1,0.35) {\tiny $x$};
\node at (7.7,5.1) {\tiny $y$};
\node at (7.5,5.7) {\tiny $z$};

\node[vertex] at (4.8,5.49) {};
\node[vertex] at (3.8,0.51) {};
\node[vertex] at (7.75,0.51) {};
\node[vertex] at (0.85,5.49) {};

\end{tikzpicture}
    \caption{Genus 1 surface with 2 boundary components. Left-right and top-bottom are identified.}
\label{gen1}
\end{figure}

\end{example}

\begin{rmk}
An optimist's conjecture would be that conversely if $A$ and $B$ are homologically smooth graded gentle algebras 
      which are derived equivalent, then there exists a homeomorphism $\phi: \Sigma_A \to \Sigma_B$ inducing a 
	bijection $\Lambda_A \to \Lambda_B$ and such that $\phi_*(\eta_A)$ is homotopic to $\eta_B$. Note that to prove this, one needs to show that the topological type of $(\Sigma_A, \Lambda_A; \eta_A)$ is a derived invariant of $A$. This is encoded by the 	numerical invariants of $\eta_A$ introduced in Theorem \ref{lineorbit} (from which one can recover the topological type
	of the surface), together with
	the numbers of marked points on each boundary component.      
\end{rmk}

\begin{rmk}\label{non-smooth-rem} 
In Theorem \ref{gentle-Fuk-thm}, it is possible to drop the assumption that $A$ is smooth. Assume for simplicity
      that $A$ is proper.
      In this case, the surface $\Sigma$ would be glued together from the disks $D_f$ associated to forbidden 
      threads as before, and also disks $D_c$ with an interior hole, associated with forbidden cycles. 
      In other words, $D_c$ is an annulus whose inner boundary component has no marked points and is not glued to
      anything, while its outer boundary component is connected by strips, corresponding to the vertices in $c$, 
      to other disks (this boundary component of $D_c$ still has no stops). 
      In the presence of unmarked boundary components, there is a dual construction
      to the construction of partially wrapped Fukaya categories, 
      $\mathcal{W}(\Sigma, \Lambda; \eta)$, 
      namely, the infinitesimal wrapped Fukaya categories $\mathcal{F}(\Sigma, \Lambda; \eta)$, 
      originally introduced for general symplectic manifolds in \cite{NZ} and studied in the case of surfaces in \cite{LP-pwrap}. Its objects are graded Lagrangians 
      which do not end on the unmarked components of the boundary. 
      Thus, for non-smooth proper gentle algebras, a version of Theorem \ref{gentle-Fuk-thm} should 
      state the equivalence
	\[ D(A) \simeq  \mathcal{F}(\Sigma_A, \Lambda_A; \eta_A) \]
	However, we have not checked that the collection of Lagrangians $\{L_v\}$ 
      given by the construction in Theorem \ref{gentle-Fuk-thm} (and modified as above) generates 
      $\mathcal{F}(\Sigma_A, \Lambda_A; \eta_A)$. 
\end{rmk}

\begin{rmk} We note that the statement of Theorem \ref{gentle-Fuk-thm} is mentioned in Section 3.4 of \cite{HKK}. In the special case when the gentle algebra is trivially graded, the construction of the surface $\Sigma$ and the dual set of Lagrangians to $\{ L_i\}$ appeared again in \cite{OPS} after this work was posted on arXiv. The authors of \cite{OPS} work with the Kozsul dual gentle algebra from a representation theoretical perspective. From the point of view of \cite{HKK}, these Koszul dual algebras can be understood as the infinitesimal Fukaya categories as explained in the previous remark. Note also that when every boundary component has at least one stop which is equivalent to requiring that corresponding gentle algebras are homologically smooth and proper, infinitesimal and partially wrapped Fukaya categories are equivalent. We refer to \cite{EL} for general results about Koszul duality in the setting of Fukaya categories.  
\end{rmk}

\section{Derived equivalences between stacky curves}

      \subsection{Chains}
      
Recall that in \cite{LP-pwrap} we considered stacky curves $C(r_0,\ldots,r_n;k_1,\ldots,k_{n-1})$
obtained by gluing weighted projective lines \[ B(r_0,r_1),B(r_1,r_2),\ldots,B(r_{n-1},r_n) \] into a chain,
where $k_i\in (\Z/r_i)^*$ are used to determine the stacky structure of the nodes in this chain.

Here $B(a,b)$, for $a,b >0$, denotes the weighted projective line stack $(\A^2\setminus 0)/\G_m$,
where $\G_m$ acts with weights $(a,b)$ (see e.g., \cite[Sec.\ 2]{AKO} and references therein).
It has two stacky points $q_-$ and $q_+$ such that $\Aut(q_-)=\mu_a$, $\Aut(q_+)=\mu_b$.
To form the chain $C(r_0,\ldots,r_n;k_1,\ldots,k_{n-1})$, we glue the point $q_+$ in $B(r_{i-1},r_i)$ with the point $q_-$
in $B(r_i,r_{i+1})$, so that the obtained stacky node
locally looks like the quotient of $xy=0$ by
the action of $\mu_{r_i}$ of the form $\zeta\cdot (x,y)=(\zeta^{k_i} x,\zeta y)$.

Note that in \cite{STZ} similar stacky curves are considered but with all $k_i=-1$ (the corresponding stacky nodes are called
{\it balanced}).

We also allow the possibility for $r_0=0$ (resp., $r_n=0$): in
this case $B(0,r_1)$ (resp. $B(r_{n-1},0)$) denotes
the weighted affine line $\mathbb{A}^1(r_1)=B(1,r_1)\setminus\{q_-\}$ (resp. 
$\mathbb{A}^1(r_{n-1})=B(r_{n-1},1)\setminus\{q_+\}$). 

We showed in \cite[Thm. B]{LP-pwrap}
      that the bounded derived category of coherent sheaves on such a stacky curve is equivalent to
the partially wrapped Fukaya category of a surface obtained by a certain gluing of the annuli that we will now describe.

      Namely, let $A(r,r')$ denote the annulus with ordered boundary components that has $r$ marked points 
      $p^-_1,\ldots,p^-_r$ on the first
      component and $r'$ marked points $p^+_1,\ldots,p^+_{r'}$ on the second boundary component. We visualize $A(r,r')$ as a rectangle with upper and lower sides glued, the left side
containing the points $p^-_i$ and the right side containing the points $p^+_i$, where the points are ordered vertically (the index increases when moving up).

      Given a collection of permutations $\sigma_i\in \mathfrak{S}_{r_i}$, $i=1,\ldots,n-1$, we consider the surface 
      $\Sigma^{lin}(r_0,\ldots,r_n;\sigma_1,\ldots,\sigma_{n-1})$ obtained by gluing the annuli
\[ A(r_0,r_1),A(r_1,r_2),\ldots,A(r_{n-1},r_n) \] in the following way (``lin" stands for ``linear", since we place annuli
in a line). 
For each $i=1,\ldots,n-1$, $j=1,\ldots,r_i$,
      we glue a small segment of the boundary around
      the marked point $p^+_j$ in $A(r_{i-1},r_i)$ with a small segment of the boundary around the point
      $p^-_{\sigma_i(j)}$ in $A(r_i,r_{i+1})$ by attaching a strip, as in Figure \ref{strips}.

\begin{figure}[htb!]
\centering
\begin{tikzpicture}
\tikzset{vertex/.style = {style=circle,draw, fill,  minimum size = 2pt,inner sep=1pt}}
\def \radius {2.5cm}
\tikzset{->-/.style={decoration={ markings,
        mark=at position #1 with {\arrow{>}}},postaction={decorate}}}

\begin{scope}[xscale=0.9, yscale=1]

\draw (1.45,2.03) to [in=270,out=180] (1,2.5);
\draw (4.55,2.03) to [in=270,out=0] (5,2.5);

\draw (1.45,-2.03) to [in=90,out=180] (1,-2.5);
\draw (4.55,-2.03) to [in=90,out=0] (5,-2.5);

\draw (1.45,2.03) to[in=180,out=0] (4.55,0.406);
\draw (1.45,1.218) to[in=180,out=0] (4.55,-0.406);

\draw (1.45,0.406) to [in=210,out=0] (2.35,0.83);
\draw (2.95,1.33) to [in=180,out=30] (4.55,2.03);

\draw (1.45,-0.406) to [in=210,out=0] (2.85,0.3);
\draw (3.6,0.78) to [in=180,out=30] (4.55,1.218);

    \draw (1.45,-1.218) to (4.55,-1.218);
\draw (1.45,-2.03) to (4.55,-2.03);

\draw(1.45,1.218) to (1.45,0.406);
 \draw(4.55,1.218) to (4.55,0.406);
 \draw(1.45,-1.218) to (1.45,-0.406);
 \draw(4.55,-1.218) to (4.55,-0.406);

\draw (-2, -2) to (-2, 2);
\draw (-2,2) to [in=270,out=0] (-1.45,2.5);
\draw (-2,-2) to [in=90,out=0] (-1.45,-2.5);

\draw (-1.45,2.5) to (1,2.5);
\draw (-1.45,-2.5) to (1,-2.5);

\node[vertex] at (-2 , 1){};
\node[vertex] at (-2 , -1){};

\begin{scope}[xshift=8cm]

\draw (1.45,2.03) to [in=270,out=180] (1,2.5);
\draw (4.55,2.03) to [in=270,out=0] (5,2.5);

\draw (1.45,-2.03) to [in=90,out=180] (1,-2.5);
\draw (4.55,-2.03) to [in=90,out=0] (5,-2.5);

\draw (1.45,2.03) to[in=180,out=0] (4.55,0.406);
\draw (1.45,1.218) to[in=180,out=0] (4.55,-0.406);

\draw (1.45,0.406) to [in=210,out=0] (2.35,0.83);
\draw (2.95,1.33) to [in=180,out=30] (4.55,2.03);

\draw (1.45,-0.406) to [in=210,out=0] (2.85,0.3);
\draw (3.6,0.78) to [in=180,out=30] (4.55,1.218);

    \draw (1.45,-1.218) to (4.55,-1.218);
\draw (1.45,-2.03) to (4.55,-2.03);

\draw(1.45,1.218) to (1.45,0.406);
 \draw(4.55,1.218) to (4.55,0.406);
 \draw(1.45,-1.218) to (1.45,-0.406);
 \draw(4.55,-1.218) to (4.55,-0.406);

\draw (8, -2) to (8, 2);
\draw (8,2) to [in=270,out=180] (7.45,2.5);
\draw (8,-2) to [in=90,out=180] (7.45,-2.5);

\draw (7.45,2.5) to (5,2.5);
\draw (7.45,-2.5) to (5,-2.5);

\node[vertex] at (8 , 0){};

 \end{scope}

\draw(9,2.5) to (5,2.5);
\draw(9,-2.5) to (5,-2.5);

\end{scope}

\end{tikzpicture}
 \caption{Surface glued from annuli (top and bottom are identified). $(r_0,r_1,r_2,r_3)=(2,3,3,1)$, $\sigma_1=\sigma_2: (1,2,3) \to (2,1,3)$}
    \label{strips}
\end{figure}
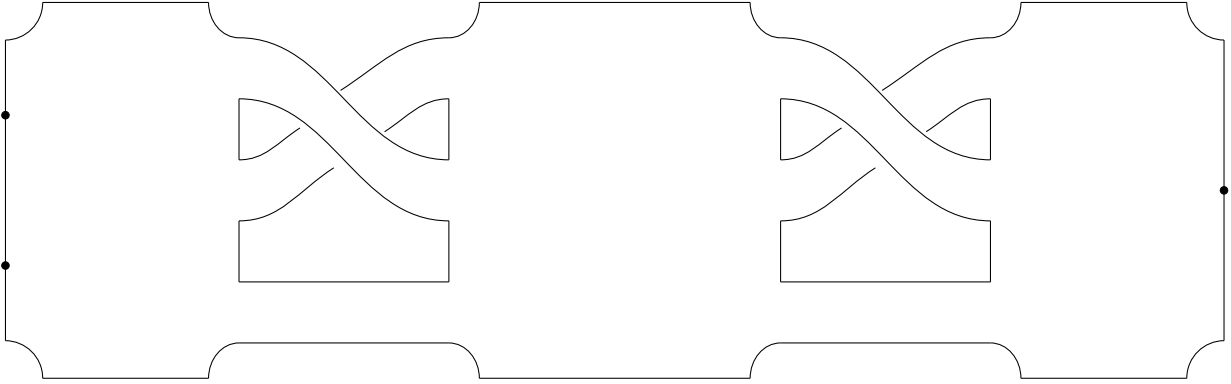

      Note that the resulting surface has two special boundary components equipped with $r_0$ and $r_n$ marked points,
      respectively (there are no other marked points on the other boundary components).
      There are also other boundary components that arise in the process of gluing. Namely, for each $i=1,\ldots,n-1$,
      the boundary components situated between the $i$th and the $(i+1)$st annuli
      are in bijection with cycles in the cycle decomposition of the commutator
      $[\sigma_i,\tau]\in \mathfrak{S}_{r_i}$, where $\tau$ is the cyclic permutation $j\mapsto j-1$.
      
      We equip $\Sigma^{lin}(r_0,\ldots,r_n;\sigma_1,\ldots,\sigma_{n-1})$ with a line field $\eta$ that corresponds to
      the horizontal direction in Figure \ref{strips}. Note that its restriction to
      each annulus is the standard line field that has zero winding numbers on both boundary components
      (this is the same choice of a line field that was made in \cite[Sec.\ 2]{LP-pwrap}).
            
      It is easy to see that the winding numbers of $\eta$ on boundary components are given as follows. For the two special boundary  components the winding numbers are equal to zero. For a boundary component corresponding to a $k$-cycle in the cycle decomposition of
      $[\sigma_i,\tau]$, the winding number is $-2k$.
      
      We are going to prove that in fact all winding numbers associated with $\eta$ are even.
      For this it is useful to construct a graph
      $$\Ga(r_1,\ldots,r_{n-1})\sub \Sigma^{lin}(r_0,\ldots,r_n;\sigma_1,\ldots,\sigma_{n-1}),$$
      which is a homotopy retract of the surface. Namely, we take one vertex in the interior of each annulus: this gives
      us $n$ vertices $v_1,\ldots,v_n$. Then we add a loop $\ga_i$ at each $v_i$, corresponding to the vertical
      circle in the $i$th annulus. Then for each of the $r_i$ strips connecting the $i$th annulus with the $(i+1)$st we add
      an edge from $v_i$ to $v_{i+1}$. 
      
      \begin{lem}\label{lin-surface-sigma-inv-lem}
      One has $[w_\eta]^{(2)}=0$, i.e., all winding numbers of $\eta$ are even.
      \end{lem}
      
      \begin{proof} 
      The embedding of the graph $\Ga(r_1,\ldots,r_{n-1})$ into 
      $\Sigma^{lin}(r_0,\ldots,r_n;\sigma_1,\ldots,\sigma_{n-1})$ induces an isomorphism on homology.
      Hence, $H_1(\Sigma^{lin}(r_0,\ldots,r_n;\sigma_1,\ldots,\sigma_{n-1}))$ is spanned by the loops $\ga_i$ together with
      the loops formed by pairs of edges connecting $v_i$ with $v_{i+1}$. The latter loops can have plane projections of one
      of the two types: they look either like circles or like figure eight curves, depending on whether the projections of the
      corresponding edges cross or not. The winding number of $\eta$ along a circle in the plane
      is $-2$, while the winding number along a figure eight curve is $0$. Since $\eta$ is constant along vertical lines,
      its winding numbers along $\ga_i$ are $0$. Now the result follows from the fact that
      $[w_\eta]^{(2)}$ is a homomorphism.
      \end{proof}
      
     To get the surface related to the stacky curve $C(r_0,\ldots,r_n;k_1,\ldots,k_{n-1})$, we now take 
      permutations $\sigma_i$ of a special kind. Namely, for each $i=1,\ldots,n-1$, we consider the permutation 
      \begin{equation}\label{sigma-ki-eq}
      \sigma_i:x\mapsto -k_ix
      \end{equation}
      of $\Z/r_i\Z$.
      We denote the resulting surface by $\Sigma^{lin}(r_0,\ldots,r_n;k_1,\ldots,k_{n-1})$. We equip it with $r_0$ and
      $r_n$ stops on two special boundary components, and denote this set of stops as $\Lambda_{r_0,r_n}$.
      Now \cite[Thm.\ B]{LP-pwrap} states that
     \[ D^b(\Coh C(r_0,\ldots,r_n;k_1,\ldots,k_{n-1})) \cong \mathcal{W}(\Sigma^{lin}(r_0,\ldots,r_n;k_1,\ldots,k_{n-1}), 
     \Lambda_{r_0,r_n}; \eta). \]
      For example, taking $r_0=r_n=0$, which corresponds to replacing the first and last weighted projective line
      by weighted affine lines, we will get the fully wrapped Fukaya categories (with no stops).
      
      Note that for each $i$ the commutator
      $[\sigma_i,\tau]$ is given by $x\mapsto x+k_i+1 \mod(r_i)$, so its cycle decomposition has $p_i=\gcd(k_i+1,r_i)$ 
      cycles of length $r_i/p_i$. Thus, the boundary winding numbers of $\eta$ on $\Sigma^{lin}(r_0,\ldots,r_n;k_1,\ldots,k_{n-1})$
      are 
      \begin{itemize}
      \item
      $0$ on each of the two special boundary components (that have marked points); 
      \item
      for each $i=1,\ldots,n-1$, the winding number $-2r_i/p_i$ repeated $p_i$ times.
      \end{itemize}
      The genus of the surface $\Sigma^{lin}(r_0,\ldots,r_n;k_1,\ldots,k_{n-1})$ is given by
      $$g=\frac{1}{2}\sum_{i=1}^{n-1}(r_i-p_i).$$
      
      Now we are going to apply Corollary \ref{lineorbit-cor} to construct examples of 
      different data $(r_0,\ldots,r_n;k_1,\ldots,k_{n-1})$ that lead
      to surfaces which are homeomorphic in a way preserving the marked points on boundary components and the line fields.      
      This will give equivalences between corresponding partially wrapped Fukaya categories and hence
      between the corresponding derived categories of stacky curves.
      
      \begin{thm}\label{stacky-der-eq-thm1} 
      The graded surface with stops $\Si^{lin}(r_0,\ldots,r_n;k_1,\ldots,k_{n-1})$ is determined up to a graded homeomorphism by the unordered pair of numbers $(r_0, r_n)$ and by the unordered collection of numbers
            \begin{equation}\label{ri/pi-seq}
            ((r_1/p_1)^{p_1},\ldots,(r_{n-1}/p_{n-1})^{p_{n-1}}),
            \end{equation}
           where $(r_i/p_i)^{p_i}$ denotes the number $r_i/p_i$ repeated $p_i$ times. 
            Hence, the same data determines the category
            $D^b(\Coh C(r_0,\ldots,r_n;k_1,\ldots,k_{n-1}))$ up to equivalence.
      \end{thm}
      
      \begin{proof} The two special components (that have stops on them) are the only ones that have the winding number $0$.
      The winding numbers of all the other boundary components are determined by the sequence \eqref{ri/pi-seq}.
      Thus, our claim is that our graded surfaces are determined by their boundary invariants (numbers of points on components
      and winding numbers). 
      We want to deduce this from Corollary \ref{lineorbit-cor}.
      
      In the case when genus is $0$, there is nothing more to check.
      In the case when genus is $\ge 2$, we observe
      that by Lemma \ref{lin-surface-sigma-inv-lem}, the invariant $\si$ vanishes for our line field $\eta$.
      On the other hand, because of the two special components with the winding number $0$, the Arf-invariant does not appear,
      so we are done in this case. 
      
      Finally, if the surface $\Si^{lin}(r_0,\ldots,r_n;k_1,\ldots,k_{n-1})$ has genus $1$ then we claim that $\wt{A}(\eta)=2$.
      Indeed, by Lemma \ref{lin-surface-sigma-inv-lem}, $\wt{A}(\eta)$ is even, so this follows from the existence of a boundary
       component with the winding number $0$.
      \end{proof}
      
      \subsubsection{Merging stacky nodes into one}
      
      Note that if $k_i=-1$ for some $i$ (which means that the corresponding node on the stacky curve is {\it balanced})
      then gluing of $A(r_{i-1},r_i)$ with $A(r_i, r_{i+1})$ results in $r_i$ boundary components on which $\eta$
      has the winding number $-2$.
     Thus, if $I\sub [1,n-1]$ is a subset of indices $i$ such that $k_i=-1$, then
      setting $r_I=\sum_{i\in I} r_i$, we get a graded homeomorphism
      $$\Sigma^{lin}(r_0,\ldots,r_n;k_1,\ldots,k_{n-1})\simeq
      \Sigma^{lin}(r_0,r_I,(r_i)_{i\not\in I},r_n;-1,(k_i)_{i\not\in I}).$$

      \begin{cor}\label{merging-lin-cor} 
      Let $I\sub [1,n-1]$ is a subset such that $k_i=-1$ for $i\in I$, and let $r_I=\sum_{i\in I} r_i$. Then there is an equivalence
      $$D^b(\Coh C(r_0,\ldots,r_n;k_1,\ldots,k_{n-1}))\simeq D^b(\Coh C(r_0,r_I,(r_i)_{i\not\in I},r_n;-1,(k_i)_{i\not\in I})).$$
      \end{cor}
      
      In the particular case $I=[1,n-1]$ (corresponding to surfaces of genus $0$), the derived equivalence of the above Corollary,
      $$D^b(\Coh C(r_0,\ldots,r_n;-1,\ldots,-1))\cong D^b(\Coh C(r_0,r_1+\ldots+r_{n-1},r_n;-1)),$$ 
      was proved in \cite{sibilla}.
      
      Note that the surfaces $\Sigma^{lin}(r_\bullet;k_\bullet)$ can have genus $1$ only when $r_{i_0}-p_{i_0}=2$
      for some $i_0\in [1,n-1]$ and $r_i=p_i$ for $i\neq i_0$. This can happen only when either $r_{i_0}=3$ or 
      $r_{i_0}=4$ and $k_{i_0}=1$. These cases are distinguished by the presence of the boundary components
      with the winding number either $-6$ or $-4$. So in the cases when the genus is $0$ or $1$ 
      we do not get any other derived equivalences between our stacky chain curves except those due to
      merging of balanced nodes.
      
      In higher genus we can sometimes merge unbalanced nodes as well.
      For example, if $\gcd(k+1,r)=1$ then for any divisor $d$ of $k+1$, $d$ stacky nodes of type $(r;k)$
      can be merged into one stacky node of type $(dr;k)$.
      
      \begin{cor}\label{merging-unb-lin-cor} 
      Assume that  for $k\in \Z_r^*$ one has $\gcd(k+1,r)=1$, and let $d>0$ be a divisor of $k+1$. 
      Then we have an equivalence
      $$D^b(\Coh C(r_0,(r)^d,r_{d+1},\ldots,r_n;(k)^d,k_{d+1},\ldots,k_{n-1}))\simeq
      D^b(\Coh C(r_0,dr,r_{d+1},\ldots,r_n;k,k_{d+1},\ldots,k_{n-1})).$$
      \end{cor}
      
      \begin{proof} We have $\gcd(k+1,dr)=d\cdot \gcd(\frac{k+1}{d},r)=d$. Thus, the pair $(dr;k)$ contributes $dp$ boundary
      components with the winding number $-2r$, which is the same as the contribution of $d$ pairs $(r;k)$.
      \end{proof}
      
     \subsubsection{Derived equivalent quotients of the coordinate cross}  
      
      To get a more interesting derived equivalence in the case of genus $\ge 2$, let us 
      specialize to the case $n=2$, $r_0=r_2=0$, $r_1=r$. Note that the corresponding stacky
      curve $C(0,r,0;k)$ is the global quotient of the affine coordinate cross $xy=0$ by the $\mu_r$-action 
      $\zeta\cdot (x,y)=(\zeta^k x,\zeta y)$. We obtain the following derived equivalences between
      these affine stacky curves. 
      
      \begin{cor} For $k,k'\in(\Z/r)^*$, 
      such that $\gcd(k+1,r)=\gcd(k'+1,r)$, there exists an equivalence
      $$D^b\Coh(C(0,r,0;k))\simeq D^b\Coh(C(0,r,0;k')).$$
      \end{cor}
      
      Note that if $k\cdot k'\equiv 1\mod r$ then we have an isomorphism $C(0,r,0;k)\simeq C(0,r,0;k')$ induced by
the involution $(x,y)\mapsto (y,x)$ on the  coordinate cross. The simplest example of a nontrivial derived equivalence
of this kind is when $r=5$, $k=1$ and $k'=2$. It would be interesting to explain this derived equivalence in a purely
algebro-geometric way. Our guess is that this can be done using the variation of GIT quotient technique.
    
 \subsection{Rings} 

      Now let us consider another class of stacky curves considered in \cite{LP-pwrap}, denoted by
      $R(r_1,\ldots,r_n;k_1,\ldots,k_n)$. They are defined by gluing the weighted projective lines
      $B(r_1,r_2),B(r_2,r_3),\ldots,B(r_n,r_1)$ into a ring, where as before $k_i\in (\Z/r_i)^*$ 
      are used to determine the stacky structure of the nodes. Thus, the point $q_+$ in $B(r_{i-1},r_i)$
      is glued with the point $q_-$ in $B(r_i,r_{i+1})$ for all $i\in \Z/n$.
      
      On the symplectic side we modify our definition of the surfaces
      $\Sigma^{lin}(r_0,\ldots,r_n;\sigma_1,\ldots,\sigma_{n-1})$ as follows. Starting with
      the annuli $A(r_1,r_2),A(r_2,r_3),\ldots,A(r_n,r_1)$ we now glue them circularly using permutations
      $\sigma_1,\ldots,\sigma_n$, so that $A(r_{i-1},r_i)$ is connected by $r_i$ strips with $A(r_i,r_{i+1})$, for $i\in\Z/n$. 
      Thus, the corresponding surface could be represented similarly to Figure
      \ref{strips} but with the right and left ends identified (so that the corresponding boundary components
      disappear). We denote the resulting surface by $\Sigma^{cir}(r_1,\ldots,r_n;\sigma_1,\ldots,\sigma_n)$.
      
      Similarly to the case of a linear gluing it is equipped with a natural line field $\eta$ that corresponds to
      the horizontal direction when the surface is depicted as on Figure \ref{strips}.
      As before, the winding numbers of $\eta$ on the boundary component corresponding to a $k$-cycle in $[\si_i,\tau]$
      is equal to $-2k$. 
      
      The analog of the graph $\Ga(r_1,\ldots,r_{n-1})$ for circular gluing is given by the graph
      $$\Ga^{cir}(r_1,\ldots,r_n)\sub\Sigma^{cir}(r_1,\ldots,r_n;\sigma_1,\ldots,\sigma_n)$$
      that still has $n$ vertices $v_1,\ldots,v_n$, a loop $\ga_i$ at each $v_i$, and $r_i$ edges connecting $v_i$ to $v_{i+1}$,
      for $i\in \Z/n$. This graph is a homotopy retract of the surface, so we can calculate the homology just by analyzing loops
      in $\Ga^{cir}(r_1,\ldots,r_n)$. In particular, we see that the homology is spanned by the loops $\ga_i$, the loops
      formed by pairs of edges connecting $v_i$ with $v_{i+1}$, and by one more loop $\beta$ corresponding
      to a horizontal line in Figure \ref{strips} 
      The analog of Lemma \ref{lin-surface-sigma-inv-lem} still holds in this case and is proved similarly:
      all winding numbers of $\eta$ are even. Note that winding number along $\beta$ is zero 
      since $\eta$ is constant along horizontal lines.
      
      As before, we specialize to the case of permutations of the form \eqref{sigma-ki-eq} and denote the corresponding
      surface by $\Sigma^{cir}(r_1,\ldots,r_n;k_1,\ldots,k_n)$. 
       The boundary winding numbers of $\eta$ on this surface are calculated as before (but now we do not have two special
       boundary components). The genus of this surface is given by 
      $$g=1+\frac{1}{2}\sum_{i=1}^{n}(r_i-p_i),$$
where $p_i=\gcd(k_i+1,r_i)$.

            By \cite[Thm. B]{LP-pwrap}, we have an equivalence
      $$D^b(\Coh R(r_1,\ldots,r_n;k_1,\ldots,k_n)) \cong \mathcal{W}(\Sigma^{cir}(r_1,\ldots,r_n;k_1,\ldots,k_n);\eta).$$
      As before, we can use Corollary \ref{lineorbit-cor} to get derived equivalences between the corresponding stacky curves.

      We have $\si(\eta)=0$ for our line field, so in the case when $r_i/p_i$ is odd for all $i$, the corresponding quadratic
      form $\ov{q}_\eta$ on $\Z_2^{2g}$ is well defined and we have to calculate its Arf-invariant.
      
      \begin{definition} For a permutation $\si\in S_d$, let us consider the vector space $V(\si)$ over $\Z_2$ with the
      basis $\a_1,\ldots,\a_d$, the even pairing such that $\a_i\cdot \a_j$ for $i<j$ is given by  
      $$\a_i\cdot\a_j=\begin{cases} 0 & \si(i)<\si(j), \\ 1 & \si(i)>\si(j),\end{cases}$$
      and the unique quadratic form $q_{\si}$ compatible with this pairing such that $q_{\si}(\a_i)=0$ for all $i$.
      Let $\ov{V}(\si)$ be the quotient of $V(\si)$ by the kernel of the pairing. If the restriction of $q_\si$ to the kernel
      of the pairing is zero then $q_\si$ descends to a quadratic form $\ov{q}_\si$ on $\ov{V}(\si)$.
      \end{definition}

\begin{lem}\label{Arf-sum-lem}
(i) For $k\in \Z_r^*$, let us consider the permutation $\si_r(k)$ of $\Z_r\setminus\{0\}=\{1,\ldots,r-1\}$ given by
$x\mapsto -kx$. Then the corresponding quadratic form $q(r,k):=q_{\si_r(k)}$ is trivial
on the kernel of the pairing on $V(\si_r(k))$ if and only if $r/p$ is odd, where $p=\gcd(k+1,r)$.  
If this is the case then $q(r,k)$ descends
to a nondegenerate quadratic form $\ov{q}(r,k)$ on the $(r-p)$-dimensional space $\ov{V}(\si_r(k))$.

\noindent
(ii) Consider the standard line field $\eta$ on the surface $\Sigma^{cir}(r_1,\ldots,r_n;k_1,\ldots,k_n)$.
Assume that all $r_i/p_i$ are odd, where $p_i=\gcd(k_i+1,r_i)$. Then
the quadratic form $\ov{q}_\eta$ is well defined, and is 
the direct sum of $\ov{q}(r_i,k_i)$ over $i=1,\ldots,n$, and the form $x^2+y^2+xy$ on $\Z_2\oplus\Z_2$.
Hence, in this case
\begin{equation}\label{circular-Arf-invariant-eq}
A(\eta)=\sum_{i=1}^n A(\ov{q}(r_i,k_i))+1 \mod (2).
\end{equation}
\end{lem}

\begin{proof}
      We are going to study the quadratic form associated with the the surface $\Si=\Si^{cir}(r;k)$.
      Let us look at the simple curves $\a_i$, $i=1,\ldots,r-1$, on $\Si$, depicted on Figure \ref{circglue}.
      In addition, we have two simple curves, $\b$ and $\ga$, 
      corresponding to a horizontal and a vertical line on Figure \ref{circglue}.
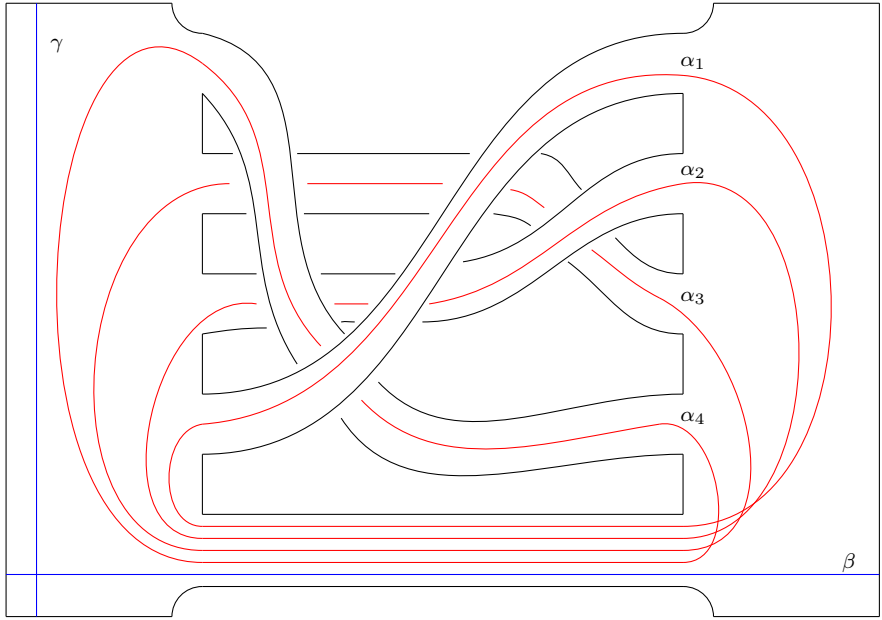
\begin{figure}[htb!]
\centering
\begin{tikzpicture}
\tikzset{vertex/.style = {style=circle,draw, fill,  minimum size = 2pt,inner        sep=1pt}}
\def \radius {2.5cm}
\tikzset{->-/.style={decoration={ markings,
        mark=at position #1 with {\arrow{>}}},postaction={decorate}}}

\begin{scope}[xscale=0.9, yscale=0.8]

\draw (1.45,4.5) to [in=270,out=180] (1,5);
\draw (8.55,4.5) to [in=270,out=0] (9,5);
\draw (1.45,-4.7) to [in=90,out=180] (1,-5.2);
\draw (8.55,-4.7) to [in=90,out=0] (9,-5.2);

\draw (-1.45, -5.2) to (-1.45, 5);
\draw (-1.45,5) to (1,5);
\draw (-1.45,-5.2) to (1,-5.2);

\draw[blue] (-1, -5.2) to (-1, 5);

\draw (11.45, -5.2) to (11.45, 5);
\draw (11.45,5) to (9,5);
\draw (11.45,-5.2) to (9,-5.2);

\draw (1.45,4.5) to[in=130,out=-15] (3.55,-0.5);
\draw (1.45,3.5) to[in=120,out=-50] (2.85,-1);
\draw (8.55,-1.5) to[in=-50,out=180] (4.05,-1.3);
\draw (8.55,-2.5) to[in=-60,out=180] (3.5,-1.9);
\draw (1.45,-1.5) to[in=180,out=0] (8.55,4.5);
\draw (1.45,-2.5) to[in=180,out=0] (8.55,3.5);
\draw (7.55,1.1) to[in=180,out=-50] (8.55,0.5);
\draw (6.85,0.7) to[in=180,out=-40] (8.55,-0.5);

\draw (6.45,2.5) to[in=130,out=-10] (7.05,1.9);
\draw (5.75,1.5) to[in=140,out=-10] (6.3,1.3);

\draw (5.4,2.5) to[in=0,out=180] (2.85,2.5);
\draw (4.8,1.5) to[in=0,out=180] (2.95,1.5);

\draw (1.9,2.5) to[in=0,out=180] (1.45,2.5);
\draw (2.1,1.5) to[in=0,out=180] (1.45,1.5);

\draw (5.3,0.7) to[in=180,out=10] (8.55,2.5);
\draw (4.7,-0.3) to[in=180,out=0] (8.55,1.5);

\draw (4.3,0.5) to [in=0,out=180] (3.2,0.5);
\draw (3.7,-0.3) to [in=10,out=180] (3.5,-0.3);

\draw (2.25,0.5) to [in=0,out=180] (1.45,0.5);
\draw (2.4,-0.4) to [in=10,out=180] (1.45,-0.5);

\draw [red] (1.45,-3.7) to [in=190,out=180] (1.45,-2);
\draw [red] (1.45,-3.9) to [in=170,out=180] (2.25,0);
\draw [red] (1.45,-4.1) to [in=180,out=180] (1.85,2);
\draw [red] (1.45,-4.3) to [in=140,out=180] (1.45,4);

\draw [red] (8.55,-3.7) to [in=-5,out=0] (8.55,3.8);
\draw [red] (8.55,-3.9) to [in=10,out=0] (8.55,2);
\draw [red] (8.55,-4.1) to [in=-30,out=0] (8.2,0.1);
\draw [red] (8.55,-4.3) to [in=10,out=0] (8.2,-2);

\draw [red] (1.45,-2) to [in=175,out=5] (8.55,3.8);

\draw [red] (3.4,0) to (3.9,0);
\draw [red] (4.8,0) to [in=190,out=10] (8.55,2);

\draw [red] (3,2) to (5,2);
\draw [red] (6,1.9) to [in=140,out=-10](6.5,1.6);
\draw [red] (7.2, 0.9) to [in=150,out=-40] (8.2,0.1);

\draw [red] (1.45,4) to [in=130,out=-40] (3.2,-0.7);
\draw [red] (3.8,-1.6) to [in=190,out=-50] (8.2,-2);

\draw (1.45,-3.5) to (8.55,-3.5);
\draw[red] (1.45,-3.7) to (8.55,-3.7);
\draw[red] (1.45,-3.9) to (8.55,-3.9);
\draw[red] (1.45,-4.1) to (8.55,-4.1);
\draw[red] (1.45,-4.3) to (8.55,-4.3);
\draw[blue] (-1.45,-4.5) to (11.45,-4.5);
\draw (1.45,-4.7) to (8.55,-4.7);

\draw(1.45,2.5) to (1.45,3.5);
\draw(1.45,0.5) to (1.45,1.5);
\draw(1.45,-1.5) to (1.45,-0.5);
\draw(1.45,-3.5) to (1.45,-2.5);

\draw(8.55,2.5) to (8.55,3.5);
\draw(8.55,0.5) to (8.55,1.5);
\draw(8.55,-1.5) to (8.55,-0.5);
\draw(8.55,-3.5) to (8.55,-2.5);

\node at (8.7, 4) {\tiny $\alpha_1$};
\node at (8.7, 2.2) {\tiny $\alpha_2$};
\node at (8.7, 0.1) {\tiny $\alpha_3$};
\node at (8.7, -1.9) {\tiny $\alpha_4$};

\node at (11, -4.3) {\tiny $\beta$};

\node at (-0.7,4.3) {\tiny $\gamma$};

\end{scope}

\end{tikzpicture}
    \caption{Circular gluing with $r=5$, $k=1$ (left-right, top-bottom are identified).}
    \label{circglue}
\end{figure}

      Using the graph $\Ga^{cir}(r)$, we see that the classes $[\b]$, $[\ga]$ and $([\a_i])_{i=1,\ldots,r-1}$ 
      span $H_1(\Sigma,\Z_2)$. Furthermore, the restriction of the intersection pairing to
      the subspace generated by $([\a_i])$ gives precisely the pairing on $V(\si_r(k))$. On the other hand, both $[\b]$ and $[\ga]$
      are orthogonal to this subspace and $\b\cdot \ga=1$. It follows that the kernel of the intersection
      pairing on $H_1(\Sigma,\Z_2)$ is equal to the kernel of the pairing on $V(\si_r(k))$, and the quotient of 
      $H_1(\Sigma,\Z_2)$ by this kernel is the direct sum of $\ov{V}(\si_r(k))$ and the $2$-dimensional space spanned by $[\b]$ 
      and $[\ga]$. Since the kernel of the intersection pairing
      is spanned by the classes of the $p$ boundary components and is $(p-1)$-dimensional, we deduce that 
      $\dim \ov{V}(\si_r(k))=r-p$.

      Furthermore, the winding number along each $\a_i$ is $-2$ so $q_\eta(\a_i)=0 \mod (4)$. 
      Thus, the restriction of the $\Z_2$-valued
      form $q_\eta/2$ to the subspace $V(\si_r(k))$ is precisely $q(r,k)$. It follows that $q(r,k)$ vanishes on the kernel of the pairing
      if and only if $q_\eta$ vanishes on the boundary cycles, which happens exactly when $r/p$ is odd (recall that the value
      of $q_\eta$ on any boundary cycle is $2-2r/p \mod(4)$).
      
      Next, we observe that the winding numbers
      along either $\b$ and $\ga$ is zero, so $q_\eta(\b)=q_\eta(\ga)=2$, and hence, 
      $$(q_{\eta}/2)(x\b+y\ga)=x^2+y^2+xy \mod (2).$$ 
      
      This immediately implies (i) and (ii) in the case of $\Sigma^{cir}(r;k)$. In the case of a general surface 
      $\Sigma^{cir}(r_1,\ldots,r_n;k_1,\ldots,k_n)$, the space $H_1(\Sigma,\Z_2)$ is spanned by the classes of loops 
      $(\a_{i,j})_{j=1,\ldots,r_i-1}$ connecting $i$th and $(i+1)$st annuli (defined in the same way as $(\a_j)$), 
      as well as by the classes of vertical loops $\ga_i$, one in each annulus,
      and by the horizontal loop $\b$. The restriction of $q_\eta/2$ to the set of classes $(\a_{i,j})$, for fixed $i$, 
      agrees with the form $q(r_i,k_i)$ on $V(\si_{r_i}(k_i))$.
      Furthermore, all the classes $\ga_i-\ga_j$ lie in the kernel of the intersection form,
      and the classes $\ga_i$ and $\b$ are orthogonal to $(\a_{i,j})$. It follows that the quotient of $H_1(\Sigma,\Z_2)$
      by the kernel of the intersection form splits into a direct sum of
      $\ov{V}(\si_{r_i}(k_i))$ over $i=1,\ldots,n$, and the $2$-dimensional subspace generated by the classes of the loops $\b$, 
      $\ga=\ga_1$. Now the result follows as in the case $n=1$.      

\end{proof}

We have the following analog of Theorem \ref{stacky-der-eq-thm1} for the circular gluing.

     \begin{thm}\label{stacky-der-eq-thm2}
     The graded surface $\Si^{cir}(r_1,\ldots,r_n;k_1,\ldots,k_n)$ is determined up to a graded homeomorphism 
     by the unordered collection of numbers
           $$((r_1/p_1)^{p_1},\ldots,(r_n/p_n)^{p_n}),$$
and in addition, in the case when all $r_i/p_i$ are odd, by the invariant $A(\eta)$ given by \eqref{circular-Arf-invariant-eq}.
            Hence, the same data determines the category
            $D^b(\Coh R(r_1,\ldots,r_n;k_1,\ldots,k_n))$ up to equivalence.
      \end{thm}
      
      \begin{proof} 
      As before, this follows from Corollary \ref{lineorbit-cor}.
      In the case when the genus is $\ge 2$, we know that the invariant $\si$ vanishes for our line field $\eta$, and the assertion
      follows from Lemma \ref{Arf-sum-lem}.
      
      The case of genus $1$ appears only when $k_i=-1$ for all $i$, in which case one immediately verifies that $\wt{A}(\eta)=0$.
      \end{proof}

           
      \subsubsection{Case of irreducible stacky curves}
      
      Assume that $n=1$. Using Theorem \ref{stacky-der-eq-thm2} we can find examples of different $k$ and $k'$ such that
      the surfaces $\Sigma^{cir}(r;k)$ and $\Sigma^{cir}(r;k')$ are graded homeomorphic, so we get interesting examples
      of derived equivalences between irreducible stacky curves.
      
      \begin{cor} Assume that $r\equiv 0 \mod (4)$, and $k,k'\in\Z_r^*$ are such that $k\equiv 1 \mod (4)$, $k'\equiv 1 \mod (4)$, and
      $\gcd(k+1,r)=\gcd(k'+1,r)$. Then we have an equivalence
      $$D^b(\Coh R(r;k))\simeq D^b(\Coh R(r;k')).$$
      \end{cor}
      
      \begin{proof} In this case $k+1\equiv 2\mod (4)$, so $p=\gcd(k+1,r)\equiv 2\mod (4)$ and hence, $r/p$ is even.
      It follows that the winding numbers of boundary components are divisible by $4$, so the Arf-invariant does not appear.
      \end{proof}
            
      Now let us consider the case when $k\in\Z_r^*$ satisfies $\gcd(k+1,r)=1$. 
      Note that this is possible only when $r$ is odd, and
      by Theorem \ref{stacky-der-eq-thm2}, the graded surface
      $\Sigma^{cir}(r;k)$ (that has genus $g=(r+1)/2$) depends on $k$ only through the Arf-invariant $A(\ov{q}(r,k))$.
       
       We will compute this Arf-invariant for $k=1$ and $k=2$ in Sec.\ \ref{Arf-comp-sec} below. 
       By Theorem \ref{stacky-der-eq-thm2},
       this leads to the following derived equivalence.
        
      \begin{cor} Assume that $r\ge 7$ is not divisible by $3$ and $r\equiv \pm 1 \mod(8)$. Then we have an equivalence
      $$D^b(\Coh R(r;1))\simeq D^b(\Coh R(r;2)).$$
      \end{cor}

\begin{proof}
By Lemma \ref{irred-k1-Arf-ex} below, for odd $r$, we have 
$$A(\ov{q}(r,1))={(r-1)/2 \choose 2}\mod 2.$$
On the other hand, by Lemma \ref{irred-k2-Arf-ex}, we have
$$A(\ov{q}(r,2))=(r-1)/2 \mod 2.$$
One can easily check that these two invariants are the same precisely when $r\equiv \pm 1 \mod (8)$.
\end{proof}

      \subsubsection{Merging stacky nodes}
      
      Note that the pairs $(r_i,k_i)$ with $k_i=-1$ do not contribute to the Arf-invariant $A(\eta)$ since in this
      case $\ov{V}(\si_{r_i}(-1))=0$. Thus, the analog of Corollary \ref{merging-lin-cor} still holds.
      
      \begin{cor}\label{merging-cir-cor} 
      Let $I\sub [1,n-1]$ is a subset such that $k_i=-1$ for $i\in I$, and let $r_I=\sum_{i\in I} r_i$. Then there is an equivalence
      $$D^b(\Coh R(r_1,\ldots,r_n;k_1,\ldots,k_n))\simeq D^b(\Coh R(r_I,(r_i)_{i\not\in I};-1,(k_i)_{i\not\in I})).$$
      \end{cor}
      
      One has to be more careful with finding an analog of Corollary \ref{merging-unb-lin-cor} since sometimes one has to compare
      the Arf-invariants. However, if some other winding numbers are divisible by $4$ then the Arf-invariant does not appear.
            
     \begin{cor}\label{merging-unb-circ-cor} 
      Assume that for $k\in \Z_r^*$ one has $\gcd(k+1,r)=1$, and let $d$ be a divisor of $k+1$. 
      Assume also that there exists $i>d$ such that
      $r_i/d_i$ is even. Then we have an equivalence
      $$D^b(\Coh R((r)^d,r_{d+1},\ldots,r_n;(k)^d,k_{d+1},\ldots,k_n))\simeq
      D^b(\Coh R(dr,r_{d+1},\ldots,r_n;k,k_{d+1},\ldots,k_n)).$$
      \end{cor} 
      
Now let us consider an example where Arf-invariant does appear. Namely, for odd $r$, let us consider merging of
two stacky nodes of type $(r;1)$ into one stacky node of type $(2r;1)$. It turns out that the corresponding
surfaces are homeomorphic but not necessarily graded homeomorphic.

      \begin{cor} For odd $r$, there exists a graded homeomorphism between 
      $\Sigma^{cir}(r,r;1,1)$ and $\Sigma^{cir}(2r;1)$
      if and only if $r\equiv 1 \mod(4)$. 
      Hence, for $r\equiv 1 \mod(4)$, we have an equivalence 
      $$D^b(\Coh R(r,r;1,1))\equiv D^b(\Coh R(2r;1)).$$ 
      \end{cor}
      
      \begin{proof} 
        For the graded surface $\Sigma^{cir}(r,r;1,1)$, we have 
        $$A(\eta)=2A(\ov{q}(r,1))+1=1 \mod (2).$$
On the other hand, for $\Sigma^{cir}(2r;1)$, we have 
       $$A(\eta)=A(\ov{q}(2r,1)+1=\frac{r+1}{2} \mod (2)$$
by Lemma \ref{irred-k1-Arf-ex}. Thus, the two Arf-invariants match exactly when $r\equiv 1 \mod(4)$.
      \end{proof}

       \subsection{Computation of the Arf-invariants}\label{Arf-comp-sec}
      
\begin{lem}\label{irred-k1-Arf-ex}
For odd $r$ one has 
$$A(\ov{q}(r,1))={(r-1)/2 \choose 2}\mod (2), $$
$$A(\ov{q}(2r,1)=\frac{r-1}{2}\mod (2).$$
\end{lem}

\begin{proof}
Since $\si_r(1)$ is the order reversing permutation of $\{1,\ldots,r-1\}$,
$q=q(r,1)=\ov{q}(r,1)$ is the unique quadratic form on the $\Z_2$-vector space $V=V(\si_r(1))$ with the basis $\a_1,\ldots,\a_{r-1}$,
compatible with the symplectic pairing given by 
\begin{equation}\label{q-r1-pairing-eq}
\a_i\cdot \a_j=1 \text{ for } i\neq j,
\end{equation}
and satisfying $q(\a_i)=0$ for all $i$.

It is well known that the Gauss sum
$$G(q):=\sum_{x\in V} (-1)^{q(x)}$$
is equal to $\pm 2^{(r-1)/2}$ and its sign determines the Arf-invariant.
It is easy to see that for any $x\in V$, one has
$$q(x)=(-1)^{{k\choose 2}},$$ 
where $k$ is the number of nonzero coordinates of $x$.
Thus, we have 
$$G(q)=\sum_{k=0}^{r-1} {r-1\choose k} (-1)^{{k\choose 2}}.$$
Now we observe that 
$$(-1)^{{k\choose 2}}=\frac{1-i}{2}\cdot i^k+\frac{1+i}{2}\cdot (-i)^k,$$
where $i=\sqrt{-1}$. Thus, we have
\begin{align*}
&\sum_{k=0}^{r-1} {r-1\choose k} (-1)^{{k\choose 2}}=
\frac{1-i}{2}\cdot (1+i)^{r-1}+\frac{1+i}{2}\cdot (1-i)^{r-1}=\\
&2^{(r-1)/2}\cdot [\frac{1-i}{2}\cdot i^{(r-1)/2}+ \frac{1+i}{2}\cdot (-i)^{(r-1)/2}]=2^{(r-1)/2}\cdot (-1)^{{(r-1)/2\choose 2}},
\end{align*} 
which proves our formula for the Arf-invariant of $\ov{q}(r,1)$.

Now, let us consider the form $q=q(2r,1)$ on the vector space $W=V(\si_{2r}(1))$ with the basis $\a_1,\ldots,\a_{2r-1}$,
equipped with the even pairing given by \eqref{q-r1-pairing-eq}, where $q$ is compatible with the pairing and satisfies
$q(\a_i)=0$. Note that the kernel of the pairing is spanned by the vector $v_0=\sum_{k=1}^{2r-1}\a_k$, and we have
$$q(v_0)={2r-1\choose 2}=0 \mod (2),$$ 
since $r$ is odd.
Thus, the form $q$ descends to a nondegenerate quadratic form $\ov{q}=\ov{q}(2r,1)$ on $\ov{W}=W/\lan v_0\ran$.
We claim that its Arf-invariant is
$$A(\ov{q})=\frac{r-1}{2} \mod 2.$$

Indeed, again we consider the Gauss sum
$$G(\ov{q}):=\sum_{x\in \ov{W}} (-1)^{\ov{q}(x)}.$$
We have
\begin{align*}
&G(\ov{q})=\frac{1}{2}\cdot G(q)=
\frac{1}{2}\sum_{k=0}^{2r-1} {2r-1\choose k}(-1)^{{k\choose 2}}=\\
&\frac{1-i}{4}\cdot (1+i)^{2r-1}+\frac{1+i}{4}(1-i)^{2r-1}=(-4)^{(r-1)/2}. 
\end{align*}
\end{proof}

\begin{lem}\label{irred-k2-Arf-ex}
Assume that $r$ is odd and not divisible by $3$. Then
$$A(\ov{q}(r,2))=\frac{r-1}{2} \mod (2).$$
\end{lem}

\begin{proof}
The form $q=q(r,2)=\ov{q}(r,2)$ is in $\Quad(V)$, where $V$ is the $\Z_2$-space with the basis $\a_1,\ldots,\a_{r-1}$ and the symplectic pairing given by
      $$\a_i\cdot \a_j=\begin{cases} 0, & i<j< i+ (r-1)/2, \\ 1, &\text{otherwise},\end{cases}$$
where $i<j$. Furthermore, $q$ is determined by $q(\a_i)=0$ for all $i$.
      It is easy to see that by renumbering the classes $(\a_i)$ as follows: 
      $$\a'_1=\a_{(r-1)/2},\ldots,\a'_{(r-1)/2}=\a_1,\a'_{(r-1)/2+1}=\a_{r-1},\ldots,\a'_{r-1}=\a'_{(r-1)/2+1},$$
      we get 
$$\a'_i\cdot \a'_j=\begin{cases} 1, & i<j<i+(r-1)/2, \\ & 0, j\ge i+(r-1)/2,\end{cases}$$ 
and $q\in\Quad(V)$ still satisfies $q(\a'_i)=0$ for all $i$.

We will compute the Arf-invariant by relating $(V,q)$ to another space with a quadratic form.
For every $k\ge 0$, such that $k\not\equiv 2 \mod(3)$, let us consider a $\Z_2$-vector space $W_k$ with the
basis $\b_1,\ga_1,\ldots,\b_k,\ga_k$, the even pairing given by the rule
$$\b_i\cdot \b_j=1 \text{ for } i\neq j; \ \ \ga_i\cdot \ga_j=1 \text{ for } i\neq j;$$
$$\b_i\cdot \ga_j=1 \text{ for } i\le j; \ \ \b_i\cdot \ga_j=0 \text{ for }i>j,$$
and the quadratic form $q_k$ in $\Quad(W_k)$ such that $q_k(\b_i)=q(\ga_i)=1$ for all $i$.

First, we will prove that $A(q)=A(q_{(r-1)/2-2})$ and then we will prove that
\begin{equation}\label{A-qk-eq}
A(q_k)= k \mod (2).
\end{equation}
To relate $(V,q)$ with $(W_{(r-1)/2-2},q_{(r-1)/2-2})$ let us 
consider the $2$-dimensional isotropic subspace $I\sub V$ spanned by $\a'_1$ and $\a'_{r-1}$. We have
$q|_I\equiv 0$, so the Arf-invariant of $q$ is equal to that of the induced quadratic form on $I^\perp/I$.
Now setting 
$$\ga_i=\a'_2+\a'_{2+i}, \ \ \b_i=\a'_{(r-1)/2+1}+\a_{(r-1)/2+1+i},$$
for $i=1,\ldots,(r-1)/2-2$, we get an identification of $I^\perp/I$ with $W_{(r-1)/2-2}$, compatible with the quadratic forms.
Hence, $A(q)=A(q_{(r-1)/2-2})$.

To prove \eqref{A-qk-eq} we use induction on $k$.
It is easy to check that $A(q_1)=1$ (and $A(q_0)=0$ for trivial reasons), so it is enough to
establish the formula
$$A(q_k)=A(q_{k-3})+1.$$
To this end we consider the $2$-dimensional isotropic subspace $J\sub W_k$
spanned by $\b_k+\ga_1$ and $\b_1+\b_k+\ga_k$. We have $q_k|_J=0$, and 
our formula follows from the identification
$$J^{\perp}/J\simeq W_{k-3}\oplus W_1,$$
where the standard basis of $W_{k-3}$ corresponds to the elements 
$$(\b_2+\b_{2+i} \mod J, \ga_2+\ga_{2+i} \mod J)_{1\le i\le k-3}$$
while a copy of $W_1$ spanned by $\b_k\mod J$ and $\ga_k\mod J$.
\end{proof}

\end{document}